\documentclass{amsart}
\usepackage{amsfonts,amssymb,amsmath,euscript}
    \newtheorem{thm}{Theorem}                     [section]
    \newtheorem{thm*}{Theorem}
    \newtheorem{prop}[thm]{Proposition}
    \newtheorem{lemma}[thm]{Lemma}
    \newtheorem{cor}[thm]{Corollary}

    \newtheorem{lemma*}{Lemma}    

    \newtheorem{defn}[thm]{Definition}                 
    \newtheorem{rems}[thm]{Remark}                     
    \newtheorem{rems*}{Remark}   

\newcommand{\ndef}{\newcommand*}
\def\rndef{\renewcommand}

\ndef{\myaddress}[1]{\begin{center} \it\small #1 \end{center}}

\ndef{\clA}{{\mathcal A}} \ndef{\rmA}{{\mathrm A}} \ndef{\mbA}{{\mathbb A}} \ndef{\bfA}{{\mathbf A}} \ndef{\euA}{{\EuScript A}} \ndef{\frA}{{\mathfrak A}}
\ndef{\clB}{{\mathcal B}} \ndef{\rmB}{{\mathrm B}} \ndef{\mbB}{{\mathbb B}} \ndef{\bfB}{{\mathbf B}} \ndef{\euB}{{\EuScript B}} \ndef{\frB}{{\mathfrak B}}
\ndef{\clC}{{\mathcal C}} \ndef{\rmC}{{\mathrm C}} \ndef{\mbC}{{\mathbb C}} \ndef{\bfC}{{\mathbf C}} \ndef{\euC}{{\EuScript C}} \ndef{\frC}{{\mathfrak C}}
\ndef{\clD}{{\mathcal D}} \ndef{\rmD}{{\mathrm D}} \ndef{\mbD}{{\mathbb D}} \ndef{\bfD}{{\mathbf D}} \ndef{\euD}{{\EuScript D}} \ndef{\frD}{{\mathfrak D}}
\ndef{\clE}{{\mathcal E}} \ndef{\rmE}{{\mathrm E}} \ndef{\mbE}{{\mathbb E}} \ndef{\bfE}{{\mathbf E}} \ndef{\euE}{{\EuScript E}} \ndef{\frE}{{\mathfrak E}}
\ndef{\clF}{{\mathcal F}} \ndef{\rmF}{{\mathrm F}} \ndef{\mbF}{{\mathbb F}} \ndef{\bfF}{{\mathbf F}} \ndef{\euF}{{\EuScript F}} \ndef{\frF}{{\mathfrak F}}
\ndef{\clG}{{\mathcal G}} \ndef{\rmG}{{\mathrm G}} \ndef{\mbG}{{\mathbb G}} \ndef{\bfG}{{\mathbf G}} \ndef{\euG}{{\EuScript G}} \ndef{\frG}{{\mathfrak G}}
\ndef{\clH}{{\mathcal H}} \ndef{\rmH}{{\mathrm H}} \ndef{\mbH}{{\mathbb H}} \ndef{\bfH}{{\mathbf H}} \ndef{\euH}{{\EuScript H}} \ndef{\frH}{{\mathfrak H}}
\ndef{\clI}{{\mathcal I}} \ndef{\rmI}{{\mathrm I}} \ndef{\mbI}{{\mathbb I}} \ndef{\bfI}{{\mathbf I}} \ndef{\euI}{{\EuScript I}} \ndef{\frI}{{\mathfrak I}}
\ndef{\clJ}{{\mathcal J}} \ndef{\rmJ}{{\mathrm J}} \ndef{\mbJ}{{\mathbb J}} \ndef{\bfJ}{{\mathbf J}} \ndef{\euJ}{{\EuScript J}} \ndef{\frJ}{{\mathfrak J}}
\ndef{\clK}{{\mathcal K}} \ndef{\rmK}{{\mathrm K}} \ndef{\mbK}{{\mathbb K}} \ndef{\bfK}{{\mathbf K}} \ndef{\euK}{{\EuScript K}} \ndef{\frK}{{\mathfrak K}}
\ndef{\clL}{{\mathcal L}} \ndef{\rmL}{{\mathrm L}} \ndef{\mbL}{{\mathbb L}} \ndef{\bfL}{{\mathbf L}} \ndef{\euL}{{\EuScript L}} \ndef{\frL}{{\mathfrak L}}
\ndef{\clM}{{\mathcal M}} \ndef{\rmM}{{\mathrm M}} \ndef{\mbM}{{\mathbb M}} \ndef{\bfM}{{\mathbf M}} \ndef{\euM}{{\EuScript M}} \ndef{\frM}{{\mathfrak M}}
\ndef{\clN}{{\mathcal N}} \ndef{\rmN}{{\mathrm N}} \ndef{\mbN}{{\mathbb N}} \ndef{\bfN}{{\mathbf N}} \ndef{\euN}{{\EuScript N}} \ndef{\frN}{{\mathfrak N}}
\ndef{\clO}{{\mathcal O}} \ndef{\rmO}{{\mathrm O}} \ndef{\mbO}{{\mathbb O}} \ndef{\bfO}{{\mathbf O}} \ndef{\euO}{{\EuScript O}} \ndef{\frO}{{\mathfrak O}}
\ndef{\clP}{{\mathcal P}} \ndef{\rmP}{{\mathrm P}} \ndef{\mbP}{{\mathbb P}} \ndef{\bfP}{{\mathbf P}} \ndef{\euP}{{\EuScript P}} \ndef{\frP}{{\mathfrak P}}
\ndef{\clQ}{{\mathcal Q}} \ndef{\rmQ}{{\mathrm Q}} \ndef{\mbQ}{{\mathbb Q}} \ndef{\bfQ}{{\mathbf Q}} \ndef{\euQ}{{\EuScript Q}} \ndef{\frQ}{{\mathfrak Q}}
\ndef{\clR}{{\mathcal R}} \ndef{\rmR}{{\mathrm R}} \ndef{\mbR}{{\mathbb R}} \ndef{\bfR}{{\mathbf R}} \ndef{\euR}{{\EuScript R}} \ndef{\frR}{{\mathfrak R}}
\ndef{\clS}{{\mathcal S}} \ndef{\rmS}{{\mathrm S}} \ndef{\mbS}{{\mathbb S}} \ndef{\bfS}{{\mathbf S}} \ndef{\euS}{{\EuScript S}} \ndef{\frS}{{\mathfrak S}}
\ndef{\clT}{{\mathcal T}} \ndef{\rmT}{{\mathrm T}} \ndef{\mbT}{{\mathbb T}} \ndef{\bfT}{{\mathbf T}} \ndef{\euT}{{\EuScript T}} \ndef{\frT}{{\mathfrak T}}
\ndef{\clU}{{\mathcal U}} \ndef{\rmU}{{\mathrm U}} \ndef{\mbU}{{\mathbb U}} \ndef{\bfU}{{\mathbf U}} \ndef{\euU}{{\EuScript U}} \ndef{\frU}{{\mathfrak U}}
\ndef{\clV}{{\mathcal V}} \ndef{\rmV}{{\mathrm V}} \ndef{\mbV}{{\mathbb V}} \ndef{\bfV}{{\mathbf V}} \ndef{\euV}{{\EuScript V}} \ndef{\frV}{{\mathfrak V}}
\ndef{\clW}{{\mathcal W}} \ndef{\rmW}{{\mathrm W}} \ndef{\mbW}{{\mathbb W}} \ndef{\bfW}{{\mathbf W}} \ndef{\euW}{{\EuScript W}} \ndef{\frW}{{\mathfrak W}}
\ndef{\clX}{{\mathcal X}} \ndef{\rmX}{{\mathrm X}} \ndef{\mbX}{{\mathbb X}} \ndef{\bfX}{{\mathbf X}} \ndef{\euX}{{\EuScript X}} \ndef{\frX}{{\mathfrak X}}
\ndef{\clY}{{\mathcal Y}} \ndef{\rmY}{{\mathrm Y}} \ndef{\mbY}{{\mathbb Y}} \ndef{\bfY}{{\mathbf Y}} \ndef{\euY}{{\EuScript Y}} \ndef{\frY}{{\mathfrak Y}}
\ndef{\clZ}{{\mathcal Z}} \ndef{\rmZ}{{\mathrm Z}} \ndef{\mbZ}{{\mathbb Z}} \ndef{\bfZ}{{\mathbf Z}} \ndef{\euZ}{{\EuScript Z}} \ndef{\frZ}{{\mathfrak Z}}

\ndef{\tA}{{\widetilde A}} \ndef{\tcA}{{\widetilde\clA}} \ndef{\ttcA}{\widetilde{\tcA}} \ndef{\sfA}{{\textsf A}} \ndef{\ttA}{\widetilde{\tA}} \ndef{\dzA}{{A^\sharp}}
\ndef{\tB}{{\widetilde B}} \ndef{\tcB}{{\widetilde\clB}} \ndef{\ttcB}{\widetilde{\tcB}} \ndef{\sfB}{{\textsf B}} \ndef{\ttB}{\widetilde{\tB}} \ndef{\dzB}{{B^\sharp}}
\ndef{\tC}{{\widetilde C}} \ndef{\tcC}{{\widetilde\clC}} \ndef{\ttcC}{\widetilde{\tcC}} \ndef{\sfC}{{\textsf C}} \ndef{\ttC}{\widetilde{\tC}} \ndef{\dzC}{{C^\sharp}}
\ndef{\tD}{{\widetilde D}} \ndef{\tcD}{{\widetilde\clD}} \ndef{\ttcD}{\widetilde{\tcD}} \ndef{\sfD}{{\textsf D}} \ndef{\ttD}{\widetilde{\tD}} \ndef{\dzD}{{D^\sharp}}
\ndef{\tE}{{\widetilde E}} \ndef{\tcE}{{\widetilde\clE}} \ndef{\ttcE}{\widetilde{\tcE}} \ndef{\sfE}{{\textsf E}} \ndef{\ttE}{\widetilde{\tE}} \ndef{\dzE}{{E^\sharp}}
\ndef{\tF}{{\widetilde F}} \ndef{\tcF}{{\widetilde\clF}} \ndef{\ttcF}{\widetilde{\tcF}} \ndef{\sfF}{{\textsf F}} \ndef{\ttF}{\widetilde{\tF}} \ndef{\dzF}{{F^\sharp}}
\ndef{\tG}{{\widetilde G}} \ndef{\tcG}{{\widetilde\clG}} \ndef{\ttcG}{\widetilde{\tcG}} \ndef{\sfG}{{\textsf G}} \ndef{\ttG}{\widetilde{\tG}} \ndef{\dzG}{{G^\sharp}}
\ndef{\tH}{{\widetilde H}} \ndef{\tcH}{{\widetilde\clH}} \ndef{\ttcH}{\widetilde{\tcH}} \ndef{\sfH}{{\textsf H}} \ndef{\ttH}{\widetilde{\tH}} \ndef{\dzH}{{H^\sharp}}
\ndef{\tI}{{\widetilde I}} \ndef{\tcI}{{\widetilde\clI}} \ndef{\ttcI}{\widetilde{\tcI}} \ndef{\sfI}{{\textsf I}} \ndef{\ttI}{\widetilde{\tI}} \ndef{\dzI}{{I^\sharp}}
\ndef{\tJ}{{\widetilde J}} \ndef{\tcJ}{{\widetilde\clJ}} \ndef{\ttcJ}{\widetilde{\tcJ}} \ndef{\sfJ}{{\textsf J}} \ndef{\ttJ}{\widetilde{\tJ}} \ndef{\dzJ}{{J^\sharp}}
\ndef{\tK}{{\widetilde K}} \ndef{\tcK}{{\widetilde\clK}} \ndef{\ttcK}{\widetilde{\tcK}} \ndef{\sfK}{{\textsf K}} \ndef{\ttK}{\widetilde{\tK}} \ndef{\dzK}{{K^\sharp}}
\ndef{\tL}{{\widetilde L}} \ndef{\tcL}{{\widetilde\clL}} \ndef{\ttcL}{\widetilde{\tcL}} \ndef{\sfL}{{\textsf L}} \ndef{\ttL}{\widetilde{\tL}} \ndef{\dzL}{{L^\sharp}}
\ndef{\tM}{{\widetilde M}} \ndef{\tcM}{{\widetilde\clM}} \ndef{\ttcM}{\widetilde{\tcM}} \ndef{\sfM}{{\textsf M}} \ndef{\ttM}{\widetilde{\tM}} \ndef{\dzM}{{M^\sharp}}
\ndef{\tN}{{\widetilde N}} \ndef{\tcN}{{\widetilde\clN}} \ndef{\ttcN}{\widetilde{\tcN}} \ndef{\sfN}{{\textsf N}} \ndef{\ttN}{\widetilde{\tN}} \ndef{\dzN}{{N^\sharp}}
\ndef{\tO}{{\widetilde O}} \ndef{\tcO}{{\widetilde\clO}} \ndef{\ttcO}{\widetilde{\tcO}} \ndef{\sfO}{{\textsf O}} \ndef{\ttO}{\widetilde{\tO}} \ndef{\dzO}{{O^\sharp}}
\ndef{\tP}{{\widetilde P}} \ndef{\tcP}{{\widetilde\clP}} \ndef{\ttcP}{\widetilde{\tcP}} \ndef{\sfP}{{\textsf P}} \ndef{\ttP}{\widetilde{\tP}} \ndef{\dzP}{{P^\sharp}}
\ndef{\tQ}{{\widetilde Q}} \ndef{\tcQ}{{\widetilde\clQ}} \ndef{\ttcQ}{\widetilde{\tcQ}} \ndef{\sfQ}{{\textsf Q}} \ndef{\ttQ}{\widetilde{\tQ}} \ndef{\dzQ}{{Q^\sharp}}
\ndef{\tR}{{\widetilde R}} \ndef{\tcR}{{\widetilde\clR}} \ndef{\ttcR}{\widetilde{\tcR}} \ndef{\sfR}{{\textsf R}} \ndef{\ttR}{\widetilde{\tR}} \ndef{\dzR}{{R^\sharp}}
\ndef{\tS}{{\widetilde S}} \ndef{\tcS}{{\widetilde\clS}} \ndef{\ttcS}{\widetilde{\tcS}} \ndef{\sfS}{{\textsf S}} \ndef{\ttS}{\widetilde{\tS}} \ndef{\dzS}{{S^\sharp}}
\ndef{\tT}{{\widetilde T}} \ndef{\tcT}{{\widetilde\clT}} \ndef{\ttcT}{\widetilde{\tcT}} \ndef{\sfT}{{\textsf T}} \ndef{\ttT}{\widetilde{\tT}} \ndef{\dzT}{{T^\sharp}}
\ndef{\tU}{{\widetilde U}} \ndef{\tcU}{{\widetilde\clU}} \ndef{\ttcU}{\widetilde{\tcU}} \ndef{\sfU}{{\textsf U}} \ndef{\ttU}{\widetilde{\tU}} \ndef{\dzU}{{U^\sharp}}
\ndef{\tV}{{\widetilde V}} \ndef{\tcV}{{\widetilde\clV}} \ndef{\ttcV}{\widetilde{\tcV}} \ndef{\sfV}{{\textsf V}} \ndef{\ttV}{\widetilde{\tV}} \ndef{\dzV}{{V^\sharp}}
\ndef{\tW}{{\widetilde W}} \ndef{\tcW}{{\widetilde\clW}} \ndef{\ttcW}{\widetilde{\tcW}} \ndef{\sfW}{{\textsf W}} \ndef{\ttW}{\widetilde{\tW}} \ndef{\dzW}{{W^\sharp}}
\ndef{\tX}{{\widetilde X}} \ndef{\tcX}{{\widetilde\clX}} \ndef{\ttcX}{\widetilde{\tcX}} \ndef{\sfX}{{\textsf X}} \ndef{\ttX}{\widetilde{\tX}} \ndef{\dzX}{{X^\sharp}}
\ndef{\tY}{{\widetilde Y}} \ndef{\tcY}{{\widetilde\clY}} \ndef{\ttcY}{\widetilde{\tcY}} \ndef{\sfY}{{\textsf Y}} \ndef{\ttY}{\widetilde{\tY}} \ndef{\dzY}{{Y^\sharp}}
\ndef{\tZ}{{\widetilde Z}} \ndef{\tcZ}{{\widetilde\clZ}} \ndef{\ttcZ}{\widetilde{\tcZ}} \ndef{\sfZ}{{\textsf Z}} \ndef{\ttZ}{\widetilde{\tZ}} \ndef{\dzZ}{{Z^\sharp}}

\ndef{\bfa}{{\mathbf a}}
\ndef{\bfb}{{\mathbf b}}
\ndef{\bfc}{{\mathbf c}}
\ndef{\bfd}{{\mathbf d}}

\ndef{\euu}{{\EuScript u}}

  \ndef{\eps}{\varepsilon}

\let\leq\leqslant

\ndef{\lims}[1]{\lim\limits_{#1}}
\ndef{\sums}[1]{\sum\limits_{#1}}
\ndef{\ints}[1]{\int_{#1}}
\ndef{\sups}[1]{\sup\limits_{#1}}
\ndef{\liminfty}[1]{\lims{#1\to\infty}}
\ndef{\suminf}[1]{\sums{#1=1}^\infty}

\ndef{\limo}[1]{\omega\mbox{-}\!\!\!\lims{#1\to\infty}}          
\ndef{\limL}[1]{\rmL\mbox{-}\!\!\!\lims{#1\to\infty}}            
\ndef{\limLOne}[1]{\clL_1\mbox{-}\!\lims{#1}}
\ndef{\tildelimo}[1]{\tilde\omega\mbox{-}\!\!\!\lims{#1\to\infty}}
\ndef{\slim}{\mathrm{s}\mbox{-}\!\!\lim}          
\ndef{\wlim}{\mathrm{w}\mbox{-}\!\lim}          

\ndef{\Aut}{\operatorname{Aut}}      
\ndef{\Ch}{\operatorname{ch}}        
\ndef{\End}{\operatorname{End}}      
\ndef{\Hom}{\operatorname{Hom}}      
\rndef{\ker}{\operatorname{ker}}      
\ndef{\coker}{\operatorname{coker}}      
\ndef{\im}{\operatorname{im}}        
\ndef{\Log}{\operatorname{Log}}      
\ndef{\OP}{\operatorname{OP}}        
\ndef{\Op}{\operatorname{Op}}        
\ndef{\Symb}{\operatorname{Symb}}    
\ndef{\Tr}{\operatorname{Tr}}        
\ndef{\Wres}{\operatorname{Wres}}    
\ndef{\cl}{\operatorname{cl}}        
\ndef{\com}{\operatorname{com}}
\ndef{\const}{\operatorname{const}}  
\ndef{\conv}{\operatorname{conv}}    
\rndef{\det}{\operatorname{det}}     

\ndef{\detFK}[1]{\Delta\brs{#1}} 
\ndef{\detFKrel}[2]{\Delta_{#2}\brs{#1}} 

\ndef{\adj}{\operatorname{adj}}    
\ndef{\diag}{\operatorname{diag}}    
\ndef{\dist}{\operatorname{dist}}    
\ndef{\dom}{\operatorname{dom}}      
\ndef{\ec}{\operatorname{ec}}        
\ndef{\id}{\mathrm{Id}}                        
\ndef{\ind}{\operatorname{ind}}      
\ndef{\mydeg}{\operatorname{deg}}    
\ndef{\op}{\operatorname{op}}
\ndef{\rank}{\operatorname{rank}}
\ndef{\res}{\operatorname{res}}      
\ndef{\rng}{\operatorname{ran}}      
\ndef{\sflow}{\operatorname{sf}}     
\ndef{\isf}{\operatorname{isf}}      
\ndef{\sign}{\operatorname{sign}}    
\ndef{\sgn}{\operatorname{sgn}}      
\ndef{\sing}{\operatorname{sing}}    
\ndef{\supp}{\operatorname{supp}}    
\ndef{\tr}{\operatorname{tr}}        
\ndef{\var}{\operatorname{var}}      
\ndef{\vol}{\operatorname{vol}}      
\ndef{\wn}{\operatorname{wn}}        
\ndef{\wres}{\operatorname{wres}}    
\rndef{\Im}{\operatorname{Im}}       
\rndef{\Re}{\operatorname{Re}}       

\ndef{\prng}[1]{\mathrm R_{#1}} 
\ndef{\pker}[1]{\mathrm N_{#1}} 
\ndef{\rprng}[2]{\mathrm R_{#1}^{#2}}           
\ndef{\rpker}[2]{\mathrm N_{#1}^{#2}}           
\ndef{\rsupp}[1]{\supp_r(#1)}
\ndef{\lsupp}[1]{\supp_l(#1)}
\ndef{\rslv}[1]{R_z(#1)}      
\ndef{\HH}{H}                 
\ndef{\tHH}{\tilde \HH}       
\ndef{\VV}{V}                 
\ndef{\Rz}{R_z}               
\ndef{\tRz}{\tR_z}            
\ndef{\psif}[1]{#1^{[1]}} 
\ndef{\WPlus}[1]{W_{#1}(\mbR)} 

\newcommand{\Phia}{\Phi^{(a)}}
\newcommand{\Phis}{\Phi^{(s)}}
\newcommand{\xia}{\xi^{(a)}}
\newcommand{\xis}{\xi^{(s)}}

\ndef{\bndl}{\xi}                         
\ndef{\bndlA}{\eta}                       
\ndef{\GlueMap}{\varphi}                  
\ndef{\ChartMap}{h}                       
\ndef{\chern}{\ensuremath{\mathrm{ch}}}
\ndef{\hilb}{\clH}                     
\ndef{\hilba}{\clH^{(a)}}                    
\ndef{\hilbs}{\clH^{(s)}}                    
   \ndef{\hilbasargument}{(\hilb)} 
\ndef{\LpH}[1]{\clL_{#1}\hilbasargument}       
\ndef{\saLpH}[1]{\clL_{sa}^{#1}\hilbasargument}       
\ndef{\clBH}{\clB\hilbasargument}              
\ndef{\ubBH}{\clB_1\hilbasargument}            
\ndef{\clCH}{\clC\hilbasargument}              
\ndef{\clKH}{\clK\hilbasargument}              
\ndef{\clFH}{\clF\hilbasargument}              
\ndef{\clUH}{\clU\hilbasargument}              
\ndef{\clCFH}{{\clC\clF}\hilbasargument}       
\ndef{\saBH}{\clB_{sa}\hilbasargument}         
\ndef{\saCH}{\clC_{sa}\hilbasargument}         
\ndef{\saFH}{\clF_{sa}\hilbasargument}         
\ndef{\saKH}{\clK_{sa}\hilbasargument}         
\ndef{\saCFH}{\clC\clF_{sa}\hilbasargument}    
\ndef{\clUFH}{\clU\clF\hilbasargument}         
\ndef{\Uinj}{\clU_{inj}\hilbasargument}        
\ndef{\UFinj}{\clU\clF_{inj}\hilbasargument}   

\ndef{\spproj}[2]{E^{#1}_{#2}}                      
\ndef{\spprojb}[2]{E^{#2}_{#1}}                     

\ndef{\LpN}[1]{\clL^{#1}(\clN,\tau)}     
\ndef{\saLpN}[1]{\clL^{#1}_{sa}(\clN,\tau)} 
\ndef{\rLpN}[1]{L^{#1}(\clN,\tau)}       
\ndef{\clAND}{(\clA,\clN,D)}             
\ndef{\clBA}{{\clB(\clA)}}
\ndef{\saKN}{{\clK_{sa}(\clN,\tau)}}          
\ndef{\clKN}{{\clK(\clN,\tau)}}          
\ndef{\clKtN}{{\clK(\tilde\clN,\tau)}}   
\ndef{\clFN}{{\clF(\clN,\tau)}}          
\ndef{\saFN}{{\clF_{sa}(\clN,\tau)}}     
\ndef{\clPN}{\clP(\clN)}                 
\ndef{\clQN}{\clQ(\clN,\tau)}            
\ndef{\infPN}{{\clP_\tau^\infty(\clN)}}  
\ndef{\clOF}[2]{\clF_{#1\mbox{-}#2}(\clN,\tau)}         
\ndef{\oind}[2]{{\rm \tau\mbox{-}ind}_{#1\mbox{-}#2}}   
\ndef{\tind}{\tau\mbox{-}\ind}                  
\ndef{\DInd}{\ind_{\clD,\tau}}           
\ndef{\BF}{Breuer-Fredholm}              
\ndef{\skewfred}[2]{$(#1\cdot #2)$ $\tau$\tire Fredholm}   
\ndef{\affl}{\eta}                       
\ndef{\vNa}{von Neumann algebra}         
\ndef{\nsf}{faithful normal semifinite } 
\ndef{\taubrs}[1]{\tau\brackets{#1}}     
\ndef{\sqbrs}[1]{[#1]}        
\ndef{\Sqbrs}[1]{\big[#1\big]}        
\ndef{\SqBrs}[1]{\Big[#1\Big]}        

\ndef{\domd}{\bigcap\limits_{n\ge 0} \dom\;\delta^n}         
\ndef{\DiffOP}{{\rm \clD}}
\ndef{\ADA}{\clA \cup [\clD,\clA]}
\ndef{\DixIdeal}[1]{\LpH{#1,\infty}}               
\ndef{\dixideal}{\ell^{1,\infty}}                  
\ndef{\WDixIdeal}{\LpH{1,\mathrm w}}               
\ndef{\DixIdealPos}[1]{\DixIdeal{#1}_+}            
\ndef{\DixIdealN}[1]{\LpN{#1,\infty}}              
\ndef{\DixIdealNPar}[2]{\clL^{#1,\infty}_{#2}(\clN,\tau)}    
\ndef{\DixIdealNPos}[1]{\LpN{#1,\infty}_+}                   
\ndef{\TrD}{\Tr_\omega}                                      
\ndef{\tauD}{{\tau_\omega}}                                  
\ndef{\ILogN}{\frac 1{\log(1+N)}}
\ndef{\DixNorm}[1]{\norm{#1}_{(1,\infty)}}                   
\ndef{\DixInt}[1]{\ints 0^t \mu_s(#1)\,ds}
\ndef{\DixIntL}[1]{\ints 0^{\lambda_{1/t}(#1)}\mu_s(#1)\,ds}
    \ndef{\SmallIdeal}{{\clL^{1, \mathrm w}}}
    \ndef{\SmallIdealMeas}{{\clL^{1, \mathrm w}_m}}
    \ndef{\DixIntII}[1]{\int_0^t \mu_s(#1)\,ds}
    \ndef{\DixIntf}[1]{\Phi_t(#1)}
    \ndef{\DixIntg}[1]{\Psi_t(#1)}

\ndef{\lpi}{\clL^{1,\pi}(\clN,\tau)}

\ndef{\strl}[1]{\stackrel \longrightarrow {#1}}
\ndef{\IIinfty}{$\mathrm{II}_\infty$\ }

\ndef{\fourier}[1]{\clF(#1)}          
\ndef{\HaarMeasBohrs}{\nu}            
\ndef{\BrownsMeas}{\mu}               
\ndef{\BohrCont}[1]{\tilde{#1}}       
\ndef{\APMean}{{M}}                   
\ndef{\CDSS}{{\clA_B}}                
\ndef{\matr}{{\rm Mat}}               
\ndef{\seque}[1]{\ensuremath{\{#1_n\}_{n=1}^\infty}}    
\ndef{\sequen}[2]{\ensuremath{\{#1_#2\}_{#2=1}^\infty}}    
\ndef{\Seque}[1]{\ensuremath{\left(#1_0,#1_1,#1_2,\dots\right)}}    
\ndef{\Cesaro}{H}                           
\ndef{\CesaroRPlus}{M}                      
\ndef{\Dilation}{D}                         
\ndef{\Shift}{T}                            

\ndef{\norm}[1]{\left\Vert#1\right\Vert}    
\ndef{\TrNorm}[1]{\norm{#1}_1}              
\ndef{\HSNorm}[1]{\norm{#1}_2}              
\ndef{\InftyNorm}[1]{\norm{#1}_\infty}      
\ndef{\normQN}[1]{\norm{#1}_{\clQN}}        
\ndef{\clLpnorm}[2]{\norm{#2}_{\clL^{#1}}}    
\ndef{\clLnorm}[1]{\clLpnorm{1}{#1}}    

\ndef{\ccurve}{\gamma}                      

\ndef{\abs}[1]{\left\lvert#1\right\rvert}   
\ndef{\set}[1]{\left\{#1\right\}}           
\ndef{\brackets}[1]{\left(#1\right)}        
\ndef{\brs}[1]{\brackets{#1}}               
\ndef{\Brs}[1]{\big(#1\big)}                
\ndef{\BRS}[1]{\Big(#1\Big)}                
\ndef{\scal}[2]{\left\la #1,#2\right\ra}               
\ndef{\precprec}{\prec\!\!\!\prec}
\ndef{\qeq}{\stackrel?=}
\ndef{\spectrum}[1]{\sigma_{#1}} 
\ndef{\spectruma}[1]{\sigma^{(a)}_{#1}} 
\ndef{\numrange}[1]{\mathrm{W}(#1)}                         
\rndef{\emptyset}{\varnothing}                              
\ndef{\csupp}{c}                           
\ndef{\closure}[1]{\overline{#1}}
\ndef{\linspan}[1]{\mathrm{span}\ {#1}}
\ndef{\bddborel}[1]{B(#1)}                 
\ndef{\charfunc}{\chi}
\ndef{\FrDer}{\euD}                        
\ndef{\LieDer}[1]{\pounds_{#1}\,}          
\ndef{\dds}{\left.\frac d{ds} \right|_{s = 0}}
\ndef{\ortcmp}[1]{#1^{\scriptscriptstyle \perp}}            
\ndef{\Laplace}{\Delta}                    

\ndef{\matrPQ}[3]
{
    \left(
      \begin{array}{cc}
        #1_{11} & #1_{12} \\
        #1_{21} & #1_{22}
      \end{array}
    \right)_{[#2,#3]}
}

\ndef{\margOK}{\marginpar{\bf \small OK}}

\newcounter{margcomcount}
\ndef{\margcom}[1]{\marginpar{\bf \small #1} \addtocounter{margcomcount}{1}
   \index{\indexcom{{\bf COMMENT: #1}}}}

\newcounter{margproof}
\ndef{\margproof}{\marginpar{\bf \small PROOF} \addtocounter{margproof}{1}
  \index{**** \indexcom{{\bf PROOF}}}}

\newcounter{margdetails}
\ndef{\margdetails}{\marginpar{\bf Details} \addtocounter{margdetails}{1}
  \index{**** \indexcom{{\bf DETAILS}}}}

\newcounter{margproofb}
\ndef{\margproofb}[1]{\marginpar{\bf \small Proof(B) #1} \addtocounter{margproofb}{1}
  \index{**** \indexcom{{\bf PROOF(B): #1}}}}

\newcounter{margdetailsb}
\ndef{\margdetailsb}[1]{\marginpar{\bf \small Details(B)} \addtocounter{margdetailsb}{1}
  \index{**** \indexcom{{\bf DETAILS(B): \\ #1}}}}

\newcounter{margdetailsc}
\ndef{\margdetailsc}[1]{\marginpar{\bf \small Details(C)} \addtocounter{margdetailsc}{1}
  \index{**** \indexcom{{\bf DETAILS(C): \\ #1}}}}

\newcounter{margcomcountb}
\ndef{\margcomb}[1]{\marginpar{\bf \small #1} \addtocounter{margcomcountb}{1}
   \index{\indexcom{{\bf COMMENT(B): \\ #1}}}}

\ndef{\mytimes}{\!\times\!}
\ndef{\sss}[1]{\subsubsection{}\label{#1}}

\rndef{\phi}{\varphi}
\ndef{\OpenUnitDisk}{D}
\ndef{\RHS}{RHS}                            
\ndef{\LHS}{LHS} 
\ndef{\ttt}{\Leftrightarrow}
\ndef{\then}{\Rightarrow}
\ndef{\tto}{\longrightarrow}
\ndef{\nno}{\nonumber\\}
\ndef{\newn}[1]{\index{#1} {\bfseries #1}}       
\ndef{\la}{\langle}
\ndef{\ra}{\rangle} \ndef{\dbar}{{\;\bar{\phantom{o}} \!\!\!\! d}}
\ndef{\stl}[1]{\stackrel{\vbox to 0pt{\vss\hbox{$\scriptstyle #1$}}}}
\ndef{\mathcomment}[1]{{\hfill \qquad\qquad\qquad\text{by (#1)}}}        
\ndef{\mathcomm}[1]{{\hfill \qquad\qquad\qquad\qquad\qquad\text{#1}}}        
\ndef{\details}[1]{\smallskip\begin{center} {\bf Here:}
#1\end{center}\medskip} \ndef{\indexcom}[1]{ --- #1}
\ndef{\longsim}{\ \sim \ }              
\ndef{\tire}{-}              
\ndef{\intinfinf}{\int_{-\infty}^\infty}
\ndef{\refnsftrace}{\cite[V.\,2.\,1]{TakI}} 
\ndef{\refaffloper}{\cite[IV.\,5, Exercise 3]{TakI}} 
\ndef{\refsemifinvNa}{\cite[V.\,1.\,21]{TakI}} 
\ndef{\reftaumeasurable}{\cite[Definition 1.2]{FK86PJM}} 
\ndef{\reftautraceclassaffl}{\cite[V.2, p.\,320]{TakI}} 
\ndef{\refinvoperideal}{\cite[Appendix A.2]{CP2}} 
\ndef{\reftautracenorm}{\cite[V.2, p.\,320]{TakI}} 
\ndef{\reftaucompact}{\cite{}} 
\ndef{\reftauFredholm}{\cite[Appendix B]{PR94JFA}} 

     \ndef{\npartial}{\slash\!\!\!\partial}
     \ndef{\Heis}{\operatorname{Heis}}
     \ndef{\Solv}{\operatorname{Solv}}
     \ndef{\Spin}{\operatorname{Spin}}
     \ndef{\SO}{\operatorname{SO}}
     \ndef{\Index}{\operatorname{index}}

             \ndef{\p}{\partial}
             \ndef{\dd}{|\clD|}
             \ndef{\n}{\parallel}

     \ndef{\gf}[2]{\genfrac{}{}{0pt}{}{#1}{#2}}
     \ndef{\ta}{\widetilde{\alpha}}
     \ndef{\tb}{\widetilde{\beta}}
     \ndef{\txi}{\widetilde{\xi}}
     \ndef{\tk}{\widetilde{K}}
     \ndef{\CGh}{\widetilde{\CG}}
     \ndef{\boe}{{\bf e}}\ndef{\bt}{{\bf t}}
     \ndef{\vth}{\vartheta}
     \ndef{\db}{\overline{\partial}}
     \ndef{\hV}{\hat{V}}
     \ndef{\cag}{{\clA^\Gamma}}
     \ndef{\sind}{\sigma{\rm -ind}}

\newcounter{slidecount}
\catcode`@=11
\newcommand{\slide}[2]{
  \newpage
  \addtocounter{slidecount}{1}
  \renewcommand{\@oddhead}{{\small Advanced Mathematics 1A -- 2009 \hfil Slide #1-\arabic{slidecount}}}
  \begin{center} \bf #2 \end{center}
}
\catcode`@=12

\let\LatexCite=\cite  

\let\ifnumref\iffalse 

\ndef{\ifuncited}[4]{\expandafter\ifx\csname used#4\endcsname\relax}

\ndef{\ifcited}[4]{\expandafter\ifx\csname used#4\endcsname\relax\else}

  \ndef{\papertitle}[1]{ \emph{#1}, }
  \ndef{\paperauthor}[2]{#2}  
  \ndef{\pbbi}[9]{%
      \ifcited{#1}{#2}{#3}{#5}%
        \ifnumref%
          \bibitem{#5}\paperauthor{#1}{#6},\papertitle{#7}#8.%
        \else%
          \advance #9 by 1%
          \ifnum#9<1%
            \bibitem[#4]{#5}\paperauthor{#1}{#6}, \papertitle{#7}#8.%
          \else%
            \bibitem[#4$_{\the#9}$]{#5}\paperauthor{#1}{#6},\papertitle{#7}#8.%
          \fi%
        \fi%
      \fi%
  }
  \ndef{\mbbi}[8]{%
     \ifcited{#1}{#2}{#3}{#5}%
        \ifnumref%
          \bibitem{#5}\paperauthor{#1}{#6},\papertitle{#7}#8.%
        \else%
          \bibitem[#4]{#5}\paperauthor{#1}{#6},\papertitle{#7}#8.%
        \fi%
     \fi%
  }

\ndef{\AddCite}[1]{%
   \ifuncited{0}{0}{0}{#1}%
     \expandafter\gdef\csname used#1\endcsname {}%
   \fi%
}

\def\ProcessCite#1,{%
     \ifx\relax#1%
         \let\next=\relax%
     \else%
         \AddCite{#1}%
         \let\next=\ProcessCite%
     \fi%
     \next%
}

\ndef{\AddCites}[1]{\ProcessCite#1,\relax,}

\ndef{\CiteWithoutExtension}[1]{%
   \AddCites{#1}%
   \LatexCite{#1}%
}

\def\CiteWithExtension[#1]#2{%
   \AddCites{#2}%
   \LatexCite[#1]{#2}%
}

\ndef{\CleverCite}{%
    \ifx\NChar[ %
       \let\MyCite=\CiteWithExtension %
    \else %
       \let\MyCite=\CiteWithoutExtension %
    \fi %
    \MyCite%
}

\renewcommand{\cite}{\futurelet\NChar\CleverCite}

      \ndef{\volume}[1]{{\bf #1}}
      \ndef{\VolYearPP}[3]{\ifnum#2=0 (to appear)\else\volume{#1} (#2), #3\fi}
      \ndef{\VolNoYearPP}[4]{\ifnum#3=0 (to appear)\else\volume{#1} #2 (#3), #4\fi}
      \ndef{\libcode}[1]{}

\ndef{\jnActaMath}[3]{Acta Math. \VolYearPP{#1}{#2}{#3}}                       
\ndef{\jnAdvMath}[3]{Adv. in~Math. \VolYearPP{#1}{#2}{#3}}                     
\ndef{\jnAlgAnal}[3]{Algebra i~Analiz \VolYearPP{#1}{#2}{#3}}
\ndef{\jnAmerJMath}[3]{Amer. J. Math. \VolYearPP{#1}{#2}{#3}}                  
\ndef{\jnAmerMathMonth}[3]{Amer. Math. Monthly \VolYearPP{#1}{#2}{#3}}         
\ndef{\jnAnnMath}[4]{Ann. of~Math. \VolNoYearPP{#1}{#2}{#3}{#4}}               
\ndef{\jnAnalMath}[3]{J. Anal. Math. \VolYearPP{#1}{#2}{#3}}                   
\ndef{\jnBullLondMathSoc}[3]{Bull. London Math. Soc. \VolYearPP{#1}{#2}{#3}}   
\ndef{\jnBullAMS}[3]{Bull. Amer. Math. Soc. \VolYearPP{#1}{#2}{#3}}   
\ndef{\jnCanMathBull}[3]{Canad. Math. Bull. \VolYearPP{#1}{#2}{#3}}            
\ndef{\jnCanMath}[3]{Canad. J.~Math. \VolYearPP{#1}{#2}{#3}}             
\ndef{\jnCommMathPhys}[3]{Comm. Math. Phys. \VolYearPP{#1}{#2}{#3}}             
\ndef{\jnCommPDE}[3]{Comm. Partial Differential Equations \VolYearPP{#1}{#2}{#3}}             
\ndef{\jnComptRendue}[3]{C.\,R.~Acad. Sci. Paris S\'er. A-B \VolYearPP{#1}{#2}{#3}}      
\ndef{\jnContMath}[3]{Contemporary Math. \VolYearPP{#1}{#2}{#3}}               %
\ndef{\jnDukeMJ}[3]{Duke Math. J. \VolYearPP{#1}{#2}{#3}}
\ndef{\jnDiffGeom}[3]{J.~Diff. Geom. \VolYearPP{#1}{#2}{#3}}                   
\ndef{\jnErgodicTheory}[3]{Ergodic Theory and Dynamical Systems \VolYearPP{#1}{#2}{#3}} 
\ndef{\jnFuncAnal}[3]{J.~Functional Analysis \VolYearPP{#1}{#2}{#3}}           
\ndef{\jnFunkAnalPril}[4]{Funct. Anal. Appl. \VolNoYearPP{#1}{#2}{#3}{#4}}  
\ndef{\jnGAFA}[3]{GAFA \VolYearPP{#1}{#2}{#3}}                                 
\ndef{\jnIHES}[3]{IHES Publ. Math. (Paris) \VolYearPP{#1}{#2}{#3}}             
\ndef{\jnIEOT}[3]{Integral Equations Operator Theory   \VolYearPP{#1}{#2}{#3}} 
\ndef{\jnIsrMath}[3]{Israel J.~Math. \VolYearPP{#1}{#2}{#3}}                   
\ndef{\jnKTheory}[3]{K-Theory \VolYearPP{#1}{#2}{#3}}                          
\ndef{\jnLetMathPhys}[3]{Lett. Math. Phys. \VolYearPP{#1}{#2}{#3}}             
\ndef{\jnMathAnn}[3]{Math. Ann. \VolYearPP{#1}{#2}{#3}}                        
\ndef{\jnMathAnalAppl}[3]{J.~Math. Anal. and Appl. \VolYearPP{#1}{#2}{#3}}     
\ndef{\jnMathNachr}[3]{Math. Nachr. \VolYearPP{#1}{#2}{#3}}
\ndef{\jnMathPhys}[3]{J. Math. Phys. \VolYearPP{#1}{#2}{#3}}
\ndef{\jnMathSocJap}[3]{J. Math. Soc. Japan \VolYearPP{#1}{#2}{#3}}
\ndef{\jnOperTheory}[3]{J.~Operator Theory \VolYearPP{#1}{#2}{#3}}             
\ndef{\jnPacJMath}[3]{Pacific J.~Math. \VolYearPP{#1}{#2}{#3}}                  
\ndef{\jnPositivity}[3]{Positivity \VolYearPP{#1}{#2}{#3}}
\ndef{\jnProcAmerMS}[3]{Proc. Amer. Math. Soc. \VolYearPP{#1}{#2}{#3}}         
\ndef{\jnProcCambPhilSoc}[3]{Math. Proc. Camb. Phil. Soc. \VolYearPP{#1}{#2}{#3}}
\ndef{\jnReineAngew}[3]{J.~Reine Angew. Math. \VolYearPP{#1}{#2}{#3}}          
\ndef{\jnTokyoMath}[3]{Tokyo J.~Math. \VolYearPP{#1}{#2}{#3}}
\ndef{\jnTopology}[3]{Topology \VolYearPP{#1}{#2}{#3}}
\ndef{\jnTransAmerMathSoc}[3]{Trans. Amer. Math. Soc. \VolYearPP{#1}{#2}{#3}}
\ndef{\jnIzvANSSSR}[3]{Izv. Akad. Nauk SSSR, Ser. Mat. \VolYearPP{#1}{#2}{#3}}
\ndef{\jnIzvVyshUchZav}[3]{Izv. Vyssh. Uch. Zav., Mat. \VolYearPP{#1}{#2}{#3} (Russian)}
\ndef{\jnIzdatLenUniv}[2]{Izdat. Leningrad. Univ., Leningrad, (#1), #2 (Russian)}
\ndef{\jnFieldsInsComm}[3]{Fields Inst. Comm. \VolYearPP{#1}{#2}{#3}}
\ndef{\jnDoklANSSSR}[3]{Dokl. Akad. Nauk SSSR \VolYearPP{#1}{#2}{#3}}
\ndef{\jnMatZametki}[3]{Matem. zametki \VolYearPP{#1}{#2}{#3}}
\ndef{\jnRussMathSurvey}[3]{Russian Math. Surveys \VolYearPP{#1}{#2}{#3}}
\ndef{\jnSibMathJ}[3]{Sib. Math.~J. \VolYearPP{#1}{#2}{#3}}
\ndef{\jnSovMath}[3]{J.~Soviet math. \VolYearPP{#1}{#2}{#3}}
\ndef{\jnTransMoscMathSoc}[3]{Trans. Moscow Math. Soc. \VolYearPP{#1}{#2}{#3}}
\ndef{\jnUMN}[3]{Uspekhi Mat. Nauk \VolYearPP{#1}{#2}{#3}}

\ndef{\bkTransMathMon}[2]{Trans. Math. Monographs, AMS, \volume{#1}, #2}

\ndef{\pbBirkhauser}[1]{Birkh\"auser, Boston, #1}
\ndef{\pbFactorial}[1]{Moscow, Factorial, #1}
\ndef{\pbGauthier}[1]{Gauthier-Villars, Paris, #1}
\ndef{\pbNauka}[1]{Moscow, Nauka, #1 (Russian)}
\ndef{\pbNaukaR}[1]{Москва, Наука, #1}
\ndef{\pbPrinceton}[1]{Princeton University Press, Princeton, New Jersey, #1}
\ndef{\pbPublPerish}[1]{Publish or Perish Inc., Berkeley, #1}
\ndef{\pbSpringer}[1]{Springer-Verlag, #1}

\ndef{\myauthor}[1]{\mbox{#1}}

\ndef{\Agmon}{\myauthor{Sh.\,Agmon}}
\ndef{\Ahiezer}{\myauthor{N.\,I.\,Ahiezer}}
\ndef{\Arazy}{\myauthor{J.\,Arazy}}
\ndef{\Aronszajn}{\myauthor{N.\,Aronszajn}}
\ndef{\Astashkin}{\myauthor{S.\,V.\,Astashkin}}
\ndef{\Atiyah}{\myauthor{M.\,Atiyah}}
\ndef{\Avron}{\myauthor{J.\,E.\,Avron}}
\ndef{\Azamov}{\myauthor{N.\,A.\,Azamov}}
\ndef{\Banach}{\myauthor{S.\,Banach}}
\ndef{\Benameur}{\myauthor{M-T.\,Benameur}}
\ndef{\Bennett}{\myauthor{C.\,Bennett}}
\ndef{\Berezin}{\myauthor{F.\,A.\,Berezin}}
\ndef{\Berline}{\myauthor{N.\,Berline}}
\ndef{\Birman}{\myauthor{M.\,Sh.\,Birman}}
\ndef{\Blackadar}{\myauthor{B.\,Blackadar}}
\ndef{\Bogolyubov}{\myauthor{N.\,N.\,Bogolyubov}}
\ndef{\Bonsall}{\myauthor{F.\,F.\,Bonsall}}
\ndef{\Bony}{\myauthor{J.\,F.\,Bony}}
\ndef{\BoosBavnbek}{\myauthor{B.\,Boo$\beta$-Bavnbek}}
\ndef{\Bott}{\myauthor{R.\,Bott}}
\ndef{\Branges}{\myauthor{L.\,de Branges}}
\ndef{\Bratteli}{\myauthor{O.\,Bratteli}}
\ndef{\Bredon}{\myauthor{G.\,E.\,Bredon}}
\ndef{\Breuer}{\myauthor{M.\,Breuer}}
\ndef{\Brown}{\myauthor{L.\,G.\,Brown}}
\ndef{\Bruneau}{\myauthor{V.\,Bruneau}}
\ndef{\Buslaev}{\myauthor{V.\,S.\,Buslaev}}
\ndef{\Carey}{\myauthor{A.\,L.\,Carey}}
\ndef{\CareyRW}{\myauthor{R.\,W.\,Carey}} 
\ndef{\Cartan}{\myauthor{H.\,Cartan}}
\ndef{\Chilin}{\myauthor{V.\,I.\,Chilin}}
\ndef{\Coburn}{\myauthor{L.\,A.\,Coburn}}
\ndef{\Connes}{\myauthor{A.\,Connes}}
\ndef{\Cornfeld}{\myauthor{I.\,P.\,Cornfeld}}
\ndef{\Daletskii}{\myauthor{Yu.\,L.\,Daletski\u\i}}   
\ndef{\Dixmier}{\myauthor{J.\,Dixmier}}
\ndef{\DoddsPG}{\myauthor{P.\,G.\,Dodds}}
\ndef{\DoddsTK}{\myauthor{T.\,K.\,Dodds}}
\ndef{\Douglas}{\myauthor{R.\,G.\,Douglas}}
\ndef{\Dubrovin}{\myauthor{B.\,A.\,Dubrovin}}
\ndef{\Dugundji}{\myauthor{J.\,Dugundji}}
\ndef{\Duncan}{\myauthor{J.\,Duncan}}
\ndef{\Dunford}{\myauthor{N.\,Dunford}}
\ndef{\Dykema}{\myauthor{K.\,J.\,Dykema}}
\ndef{\Edwards}{\myauthor{R.\,E.\,Edwards}}
\ndef{\Eilenberg}{\myauthor{S.\,Eilenberg}}
\ndef{\Entina}{\myauthor{S.\,B.\,\`Entina}}
\ndef{\Fack}{\myauthor{T.\,Fack}} 
\ndef{\Faddeev}{\myauthor{L.\,D.\,Faddeev}}
\ndef{\Farber}{\myauthor{M.\,Farber}}
\ndef{\Farforovskaya}{\myauthor{Yu.\,B.\,Farforovskaya}}
\ndef{\Federer}{\myauthor{H.\,Federer}}
\ndef{\Fedosov}{\myauthor{B.\,V.\,Fedosov}}
\ndef{\Figiel}{\myauthor{T.\,Figiel}} 
\ndef{\Figueroa}{\myauthor{H.\,Figueroa}}
\ndef{\Fillmore}{\myauthor{P.\,A.\,Fillmore}}
\ndef{\Fomenko}{\myauthor{A.\,T.\,Fomenko}} 
\ndef{\Fomin}{\myauthor{S.\,V.\,Fomin}}
\ndef{\Frohlich}{\myauthor{J.\,Fr\"ohlich}}
\ndef{\Fuglede}{\myauthor{B.\,Fuglede}}
\ndef{\Furutani}{\myauthor{K.\,Furutani}}
\ndef{\Gelfand}{\myauthor{I.\,M.\,Gelfand}}
\ndef{\Gesztesy}{\myauthor{F.\,Gesztesy}}     
\ndef{\Getzler}{\myauthor{E.\,Getzler}} 
\ndef{\Gilkey}{\myauthor{P.\,B.\,Gilkey}}
\ndef{\Gitler}{\myauthor{S.\,Gitler}}
\ndef{\Glazman}{\myauthor{I.\,M.\,Glazman}}
\ndef{\Glimm}{\myauthor{J.\,Glimm}}
\ndef{\Gohberg}{\myauthor{I.\,C.\,Gohberg}}
\ndef{\Goldshtein}{\myauthor{Ya.\,Goldshtein}}
\ndef{\Golze}{\myauthor{F.\,Golze}}
\ndef{\GraciaBondia}{\myauthor{J.\,M.\,Gracia-Bond\'{i}a}}
\ndef{\Greenleaf}{\myauthor{F.\,P.\,Greenleaf}}
\ndef{\Gromov}{\myauthor{M.\,Gromov}}
\ndef{\Gunning}{\myauthor{R.\,C.\,Gunning}}
\ndef{\Haagerup}{\myauthor{U.\,Haagerup}}
\ndef{\Haag}{\myauthor{R.\,Haag}}
\ndef{\Halmos}{\myauthor{Halmos}}
\ndef{\Hardy}{\myauthor{G.\,H.\,Hardy}}
\ndef{\Herbst}{\myauthor{I.\,W.\,Herbst}}
\ndef{\Higson}{\myauthor{N.\,Higson}}  
\ndef{\Hoermander}{\myauthor{L.\,Hoermander}} 
\ndef{\Hoffman}{\myauthor{K.\,Hoffman}} 
\ndef{\Ito}{\myauthor{K.\,Ito}}
\ndef{\Jaffe}{\myauthor{A.\,Jaffe}}
\ndef{\James}{\myauthor{I.\,M.\,James}}
\ndef{\Javrjan}{\myauthor{V.\,A.\,Javrjan}}
\ndef{\Jitomirskaya}{\myauthor{S.\,Jitomirskaya}}
\ndef{\Kadison}{\myauthor{R.\,V.\,Kadison}}
\ndef{\Kalton}{\myauthor{N.\,J.\,Kalton}} 
\ndef{\Kato}{\myauthor{T.\,Kato}} 
\ndef{\Kobayashi}{\myauthor{S.\,Kobayashi}}
\ndef{\Koplienko}{\myauthor{L.\,S.\,Koplienko}}
\ndef{\Korotyaev}{\myauthor{E.\,Korotyaev}}
\ndef{\Kosaki}{\myauthor{H.\,Kosaki}}
\ndef{\Kostrykin}{\myauthor{V.\,Kostrykin}}
\ndef{\Kotani}{\myauthor{S.\,Kotani}}
\ndef{\Krein}{\myauthor{Kre\u\i n}}
\ndef{\KreinMG}{\myauthor{M.\,G.\,Kre\u\i n}}
\ndef{\KreinSG}{\myauthor{S.\,G.\,Kre\u\i n}}
\ndef{\Kuroda}{\myauthor{S.\,T.\,Kuroda}}
\ndef{\Leichtnam}{\myauthor{E.\,Leichtnam}}
\ndef{\Lesch}{\myauthor{M.\,Lesch}}
\ndef{\Lesniewski}{\myauthor{A.\,Lesniewski}}
\ndef{\Levitan}{\myauthor{B.\,M.\,Levitan}}
\ndef{\Lidskii}{\myauthor{V.\,B.\,Lidskii}}
\ndef{\Lifshits}{\myauthor{I.\,M.\,Lifshits}}
\ndef{\Lindenstrauss}{\myauthor{J.\,Lindenstrauss}}
\ndef{\Loday}{\myauthor{J.-L.\,Loday}}
\ndef{\Lord}{\myauthor{S.\,Lord}}      
\ndef{\Lorentz}{\myauthor{G.\,Lorentz}}
\ndef{\Magnus}{\myauthor{W.\,Magnus}}
\ndef{\Makarov}{\myauthor{K.\,A.\,Makarov}}
\ndef{\MakarovN}{\myauthor{N.\,Makarov}}
\ndef{\Mathai}{\myauthor{V.\,Mathai}}         
\ndef{\McKean}{\myauthor{H.\,P.\,McKean}}
\ndef{\Mishchenko}{\myauthor{A.\,S.\,Mishchenko}}
\ndef{\Molchanov}{\myauthor{S.\,A.\,Molchanov}}
\ndef{\Moore}{\myauthor{C.\,C.\,Moore}}
\ndef{\Moscovici}{\myauthor{H.\,Moscovici}}  
\ndef{\Motovilov}{\myauthor{A.\,K.\,Motovilov}}
\ndef{\Moyer}{\myauthor{R.\,D.\,Moyer}}
\ndef{\Naboko}{\myauthor{S.\,N.\,Naboko}}
\ndef{\Narasimhan}{\myauthor{R.\,Narasimhan}}
\ndef{\Nomizu}{\myauthor{K.\,Nomizu}}
\ndef{\Novikov}{\myauthor{S.\,P.\,Novikov}}
\ndef{\Osterwalder}{\myauthor{K.\,Osterwalder}}
\ndef{\Patodi}{\myauthor{V.\,Patodi}}
\ndef{\Pagter}{\myauthor{B.\,de~Pagter}}  
\ndef{\Pastur}{\myauthor{L.\,A.\,Pastur}}  
\ndef{\Pavlov}{\myauthor{B.\,S.\,Pavlov}}
\ndef{\Pedersen}{\myauthor{G.\,K.\,Pedersen}}
\ndef{\Peller}{\myauthor{V.\,V.\,Peller}}
\ndef{\Perera}{\myauthor{V.\,S.\,Perera}}
\ndef{\Petunin}{\myauthor{Ju.\,I.\,Petunin}}
\ndef{\Phillips}{\myauthor{J.\,Phillips}}  
\ndef{\Piazza}{\myauthor{P.\,Piazza}}   
\ndef{\Pincus}{\myauthor{J.\,D.\,Pincus}}   
\ndef{\Poincare}{Poincar\'e}
\ndef{\Postnikov}{\myauthor{M.\,M.\,Postnikov}} 
\ndef{\Prinzis}{\myauthor{R.\,Prinzis}}
\ndef{\Privalov}{\myauthor{I.\,I.\,Privalov}}
\ndef{\Pushnitski}{\myauthor{A.\,B.\,Pushnitski}} 
\ndef{\Raeburn}{\myauthor{I.\,Raeburn}}
\ndef{\Raikov}{\myauthor{G.\,Raikov}}
\ndef{\Reed}{\myauthor{M.\,Reed}}
\ndef{\Rennie}{\myauthor{A.\,Rennie}}
\ndef{\Rickart}{\myauthor{C.\,E.\,Rickart}}
\ndef{\Riesz}{\myauthor{F.\,Riesz}}
\ndef{\Ringrose}{\myauthor{J.\,Ringrose}}
\ndef{\Rio}{\myauthor{R.\,del Rio}}
\ndef{\Robinson}{\myauthor{D.\,Robinson}}
\ndef{\Rossi}{\myauthor{H.\,Rossi}}
\ndef{\Rudin}{\myauthor{W.\,Rudin}}
\ndef{\Ruelle}{\myauthor{D.\,Ruelle}}
\ndef{\Ruzhansky}{\myauthor{M.\,Ruzhansky}}
\ndef{\Sakai}{\myauthor{Sh.\,Sakai}}
\ndef{\Sargsjan}{\myauthor{I.\,S.\,Sargsjan}}
\ndef{\Sato}{\myauthor{H.\,Sato}}
\ndef{\Schaeffer}{\myauthor{D.\,G.\,Schaeffer}}
\ndef{\Schluchtermann}{\myauthor{G.\,Schluchtermann}}
\ndef{\Schochet}{\myauthor{C.\,Schochet}}
\ndef{\SchroedingerE}{\myauthor{E.\,Schr\"odinger}}
\ndef{\Schroedinger}{\myauthor{Schr\"odinger}}
\ndef{\Schrohe}{\myauthor{E.\,Schrohe}}
\ndef{\Schwartz}{\myauthor{J.\,T.\,Schwartz}}
\ndef{\Sedaev}{\myauthor{A.\,A.\,Sedaev}}
\ndef{\Seiler}{\myauthor{R.\,Seiler}}
\ndef{\Semenov}{\myauthor{E.\,M.\,Semenov}}
\ndef{\Shabat}{\myauthor{B.\,V.\,Shabat}}
\ndef{\Shafarevich}{\myauthor{I.\,R.\,Shafarevich}}
\ndef{\Sharpley}{\myauthor{R.\,Sharpley}}
\ndef{\Shilov}{\myauthor{G.\,E.\,Shilov}}
\ndef{\Shirkov}{\myauthor{D.\,V.\,Shirkov}}
\ndef{\Shubin}{\myauthor{M.\,A.\,Shubin}}
\ndef{\Silverman}{\myauthor{H.\,Silverman}}
\ndef{\Simon}{\myauthor{B.\,Simon}}
\ndef{\Sinai}{\myauthor{Ya.\,G.\,Sinai}}
\ndef{\Singer}{\myauthor{I.\,M.\,Singer}}
\ndef{\Solomyak}{\myauthor{M.\,Z.\,Solomyak}}
\ndef{\Soloviev}{\myauthor{Yu.\,P.\,Soloviev}}
\ndef{\Spivak}{\myauthor{M.\,Spivak}}
\ndef{\Stein}{\myauthor{E.\,M.\,Stein}}
\ndef{\Stenkin}{\myauthor{V.\,V.\,Sten'kin}}
\ndef{\Stratila}{\myauthor{S.\,Stratila}}
\ndef{\Sucheston}{\myauthor{L.\,Sucheston}}
\ndef{\Sukochev}{\myauthor{F.\,A.\,Sukochev}}
\ndef{\Switzer}{\myauthor{R.\,M.\,Switzer}}
\ndef{\SzNagy}{\myauthor{B.\,Sz.-Nagy}}
\ndef{\Takesaki}{\myauthor{M.\,Takesaki}}
\ndef{\Taylor}{\myauthor{M.\,E.\,Taylor}}
\ndef{\Treves}{\myauthor{F.\,Treves}}
\ndef{\Troitsky}{\myauthor{E.\,V.\,Troitsky}}
\ndef{\Tzafriri}{\myauthor{L.\,Tzafriri}}
\ndef{\Varilly}{\myauthor{J.\,C.\,V\'{a}rilly}}
\ndef{\Vergne}{\myauthor{M.\,Vergne}}
\ndef{\Vladimirov}{\myauthor{V.\,S.\,Vladimirov}}
\ndef{\Voiculescu}{\myauthor{D.\,Voiculescu}}
\ndef{\Weiss}{\myauthor{G.\,Weiss}}
\ndef{\Wells}{\myauthor{R.\,O.\,Wells}}
\ndef{\Williams}{\myauthor{J.\,P.\,Williams}}
\ndef{\Winkler}{\myauthor{S.\,Winkler}}
\ndef{\Witten}{\myauthor{E.\,Witten}}
\ndef{\Wodzicki}{\myauthor{M.\,Wodzicki}}
\ndef{\Wojciechowski}{\myauthor{K.\,P.\,Wojciechowski}}
\ndef{\Yafaev}{\myauthor{D.\,R.\,Yafaev}}
\ndef{\Yosida}{\myauthor{K.\,Yosida}}
\ndef{\Zsido}{\myauthor{L.\,Zsido}}

\ndef{\mbCz}{\mbC^{(z)}} \ndef{\mbCr}{\mbC^{(r)}} \ndef{\yy}{y}

\ndef{\ells}{{\ell_2}} \ndef{\hlambda}{{\mathfrak h_\lambda}}
\ndef{\hlambdao}{{\mathfrak h_\lambda^{(0)}}}
\ndef{\hlambdas}{{\mathfrak h_\lambda^{(s)}}}
\ndef{\hlambdar}{{\mathfrak h_\lambda^{(r)}}} 
\begin{document}
\title[Resonance index]{Resonance index and singular \\
  spectral shift function}
\author{Nurulla Azamov}
\address{School of Computer Science, Engineering and Mathematics
   \\ Flinders University
   \\ Bedford Park, 5042, SA Australia.}
\email{azam0001@csem.flinders.edu.au} \keywords{Spectral shift
function, singular spectral shift function, resonance index}
 \subjclass[2000]{ 
     Primary 47A55; 
     Secondary 47A11
}

\begin{abstract} This paper is a continuation of my previous work on absolutely continuous and singular spectral shift functions.
Let $H_0$ be a self-adjoint operator on a framed Hilbert space $(\hilb,F)$ and let $V$ be a self-adjoint trace-class operator.
The singular part $\xis(\lambda; H_0+rV,H_0)$ of the spectral shift function is a.e. integer-valued function,
where $r \in \mbR.$
For any fixed point~$\lambda$ from the set of regular points $\Lambda(H_0,F)$ associated with the pair $(H_0,F),$
the singular spectral shift function is a locally constant function of the coupling constant $r,$
with possible jumps only at resonance points. Main result of this paper asserts that the jump of the singular spectral shift function
at a resonance point $r_0$ is equal to the so-called resonance index of the triple $(\lambda; H_{r_0},V).$

Resonance index can be described as follows. For a fixed $\lambda \in \Lambda(H_0,F),$
the resonance points $r_0$ of the path $H_r$ are real poles of a certain meromorphic function
associated with the triple $(\lambda+i0; H_{0},V).$ When $\lambda+i0$ is shifted to $\lambda+iy$ with small $y>0,$ that pole get off the real axis in the coupling constant
complex plane and, in general, splits into some $N_+$ poles in the upper half-plane and some $N_-$ poles in the lower half-plane (counting multiplicities).
Resonance index of the triple $(\lambda; H_{r_0},V)$ is the difference $N_+-N_-.$

Based on the theorem just described, a non-trivial example of singular spectral shift function is given.
\end{abstract}

\maketitle

\section{Introduction and Preliminaries}
This paper is a continuation of my previous work \cite{Az3v5} and heavily relies on it.
I give here necessary preliminaries from \cite{Az3v5}, which are necessary to state the main result of this paper.

Let $\hilb$ be a (complex separable) Hilbert space.
We denote by
$$
  R_z(H) = (H-z)^{-1}
$$
the resolvent of an operator $H$ in $\hilb.$

Recall that a \emph{frame} $F$ in a Hilbert space $\hilb$ (see \cite[Definition 2.1.1]{Az3v5})
is a Hilbert-Schmidt operator $F \colon \hilb \to \clK$ with trivial kernel and co-kernel,
where $\clK$ is another Hilbert space. The condition that co-kernel of $F$ be trivial is not essential,
since one can always replace $\clK$ by the closure of the image of~$F.$ Let
$$
  F = \sum_{j=1}^\infty \kappa_j \scal{\phi_j}{\cdot}\psi_j
$$
be Schmidt representation of~$F.$ In fact, in the definition of the frame $F$ the important data is the orthonormal basis
$(\phi_j)_{j=1}^\infty$ with a fixed order and the $\ell_2$-sequence of positive decreasing numbers $(\kappa_j).$

Let $H_0$ be a self-adjoint operator in $\hilb$ with a fixed frame $F,$
let $$V =F^*JF,$$ where $J$ is a bounded self-adjoint operator on $\clK,$ and let $$H_r = H_0+rV, \ r \in \mbR.$$
Let
$$
  T_z(H_0) = F R_{z}(H_0)F^*.
$$
We define a set
\begin{equation} \label{F: def of Lambda}
  \Lambda(H_0,F)
\end{equation}
as the set of all those real numbers $\lambda \in \mbR,$
for which the limit
$$
  \lim_{y \to 0} F R_{\lambda+iy}(H_0)F^* =: F R_{\lambda+i0}(H_0)F^*
$$
exists in uniform norm and the limit
$$
  \lim_{y \to 0} F \Im R_{\lambda+iy}(H_0)F^* =: F \Im R_{\lambda+i0}(H_0)F^*
$$
exists in trace-class norm. It is well-known that the set $\Lambda(H_0,F)$ has full Lebesgue measure (see e.g. \cite[Theorems 6.1.5 and 6.1.9]{Ya}).
It also has the following coupling constant regularity property \cite[\S 4]{Az3v5}:
$$
  \text{if} \ \lambda \in \Lambda(H_0,F), \ \text{then} \ \lambda \in \Lambda(H_r,F) \ \text{for all} \ r \notin R(\lambda; \set{H_s},V),
$$
where $R(\lambda; \set{H_s},V)$ is a discrete subset of $\mbR$ given by (see \cite[Definition 4.1.6]{Az3v5})
$$
  R(\lambda; \set{H_s},V) = \set{r \in \mbR \colon 1+rT_{\lambda+i0}(H_0)J \ \text{is not invertible}}.
$$

For every $\lambda \in \Lambda(H_0,F)$ there exists the matrix
$$
  \phi(\lambda) = \frac 1\pi \brs{\kappa_j\kappa_k \scal{\phi_j}{\Im R_{\lambda+i0(H_0)}\phi_k}},
$$
which is a non-negative trace-class operator on $\ell_2$ (see \cite[Proposition 2.4.3]{Az3v5}).
Let $\hilb_\alpha(F), \ \alpha \in \mbR,$ be the scale of Hilbert spaces associated with operator $\abs{F}$ on $\hilb$
(see e.g. \cite[section 2.6]{Az3v5} for details).
An element $f$ of $\hilb_1(F)$ has the form
$$
  f = \sum_{j=1}^\infty \beta_j \kappa_j\phi_j,
$$
where $(\beta_j) \in \ell_2.$
The \emph{evaluation operator}
$$
  \euE_\lambda(H_0) = \euE_\lambda(H_0,F) \colon \hilb_1(F) \to \ell_2
$$ is defined by formula (see \cite[\S 3]{Az3v5})
$$
  \euE_\lambda f = \sum_{j=1}^\infty \beta_j \eta_j(\lambda),
$$
where $\eta_j(\lambda)$ is the $j$-th column of the Hilbert-Schmidt matrix $\eta(\lambda) = \sqrt {\phi(\lambda)}$ divided by~$\kappa_j.$
For every $\lambda \in \Lambda(H_0,F)$ one can define
the Hilbert space
$$
  \hlambda = \hlambda(H_0) = \hlambda(H_0,F) \subset \ell_2
$$ associated with frame $F,$ as the closure of the image of the evaluation operator in~$\ell_2:$
$$
  \hlambda(H_0,F) = \overline {\euE_\lambda(\hilb_1(F))}.
$$
The family of Hilbert spaces $\set{\hlambda, \lambda \in \Lambda(H_0,F)}$ form a measurable family with measurability base $\phi_j(\lambda):=\euE_\lambda\phi_j, \ j =1,2,\ldots$
(see \cite[Lemma 3.1.1]{Az3v5}).
The corresponding direct integral of Hilbert spaces
$$
  \euH :=   \int^\oplus_{\Lambda(H_0,F)} \hlambda \,d\lambda
$$
is naturally isomorphic to the absolutely continuous part of $\hilb$ with respect to~$H_0.$
The natural isomorphism is given by the direct integral of operators $\euE_\lambda(H_0,F)$
(see \cite[Proposition 3.3.5]{Az3v5})
$$
  \euE := \int^\oplus_{\Lambda(H_0,F)} \euE_\lambda(H_0,F)\,d\lambda,
$$
More exactly, the operator $\euE \colon \hilb_1 \to \euH,$ thus defined, can be considered as an operator on $\hilb,$
and this operator is bounded on $\hilb,$ vanishes on the singular subspace of $H_0$ and maps isometrically the absolutely continuous subspace
of $H_0$ onto $\euH$ (see \cite[\S 3]{Az3v5} for details and proofs).

For any $\lambda \in \Lambda(H_0,F)$ the operator $R_{\lambda+i0}(H_0)$ can be considered as a compact operator
$$
  R_{\lambda+i0}(H_0) \colon \hilb_1(F) \to \hilb_{-1}(F)
$$
and the operator $V$ can be considered as a bounded operator
$$
  V \colon \hilb_{-1}(F) \to \hilb_{1}(F),
$$
so that for any non-resonant $r$ (that is, for $r \notin R(\lambda; \set{H_s},V)$) the operator $S(\lambda; H_r,H_0) \colon \hlambda(H_0) \to \hlambda(H_0),$
given by the well-known stationary formula
$$
  S(\lambda; H_r,H_0) = 1_\lambda - 2\pi i \euE_\lambda(H_0) rV(1+rR_{\lambda+i0}(H_0)V)^{-1} \euE_\lambda^\diamondsuit(H_0),
$$
makes sense, where $$\euE_\lambda^\diamondsuit(H_0) = \abs{F}^{-2}\euE_\lambda^*(H_0) \colon \hlambda(H_0,F) \to \hilb_{-1}(F).$$
In \cite{Az3v5} it was shown that thus defined operator is unitary and that the operator
$$
  \int^\oplus_{\Lambda(H_0,F)} S(\lambda; H_1,H_0)\,d\lambda
$$
is naturally isomorphic to the scattering operator $\bfS(H_0+V,H_0)$ of the pair $H_0,H_0+V,$
the natural isomorphism being given by $\euE.$ That is, the operator $S(\lambda; H_r,H_0)$ is in fact the scattering matrix
(or on-shell scattering operator, --- in physical language). It is also shown in \cite{Az3v5} that, for a fixed $\lambda \in \Lambda(H_0,F),$
the scattering matrix $S(\lambda; H_r,H_0)$ as a function of the coupling constant $r$ admits analytic continuation to a neighbourhood of~$\mbR.$

An essential difference between classical approach to abstract stationary scattering theory (see e.g. \cite{BE,Ya}) and the one given in \cite{Az3v5}
is that, while in the former theory the scattering matrix is defined for a.e.~$\lambda$ and it is impossible to consider the scattering
matrix at a single point, --- this just does not make sense, in the latter theory the scattering matrix is explicitly constructed for \emph{every}
$\lambda$ from an a-priori given set of full Lebesgue measure.

Let $\clU_1(\hilb)$ be the group of all unitary operators on $\hilb$ of the form $1+$ trace-class operator, with topology inherited from $\clL_1(\hilb).$
Let $U \colon [a,b] \to \clU_1(\hilb)$ be a continuous path of unitary operators such that $U(a) = 1.$
$\xi$-invariant of this path can be defined as a continuous function $\xi \colon [a,b] \to \mbR$
such that for any $r \in [a,b]$
$$
  \xi(r) = - \frac 1{2\pi i} \log \det U(r),
$$
where the branch of $\log$ is chosen so that the function in the right hand side is continuous.
In fact, $\xi$-invariant is the Pushnitski $\mu$-invariant $\mu(\theta,\lambda)$ averaged
over the angle variable $\theta,$ see \cite{Pu01FA,Az3v5}, where $\mu(\theta,\lambda)$ is an integer which --- roughly speaking --- measures
the number of eigenvalues of $U(r)$ which cross the point $e^{i\theta}$ on the unit circle as $r$ moves from $0$ to $1.$

In \cite{Az3v5} for every $\lambda \in \Lambda(H_0,F)$ there was introduced the so-called absolutely continuous part $\xia(\lambda; H_r,H_0)$ of the spectral shift function
which is the $\xi$-invariant of the path of unitary operators
$$
  [0,1] \ni r \to S(\lambda; H_r,H_0).
$$
(Absolutely continuous and singular parts of spectral shift function were first introduced in \cite{Az}).
The function $\xia(\lambda; H_r,H_0)$ is the density of the absolutely continuous measure
$$
  \Delta \mapsto \int_0^r \Tr\brs{V E_\Delta^{H_s} P^{(a)}(H_s)}\,ds,
$$
where $E_\Delta^{H_s}$ is the spectral projection of the operator $H_s = H_0 +sV$ and $P^{(a)}(H_s)$
is the projection onto the absolutely continuous subspace of $H_s.$

It was shown in \cite{Az3v5} that $\xia(\lambda; H_r,H_0)$ is a summable function which differs from the spectral shift function $\xi(\lambda; H_r,H_0)$ by an a.e. integer-valued
function denoted by $\xis(\lambda; H_r,H_0)$ and called the singular part of the spectral shift function. In \cite{Pu01FA} it was actually shown that
the spectral shift function $\xi(\lambda; H_1,H_0)$ can be defined for every~$\lambda$ from the set $\Lambda(H_0,F) \cap \Lambda(H_1,F)$
of full Lebesgue measure as the $\xi$-invariant of a continuous path of unitary operators introduced in \cite{Pu01FA},
though this result was not formulated in \cite{Pu01FA} in these terms. So, for every fixed $\lambda \in \Lambda(H_0,F)$
the spectral shift function $\xi(\lambda; H_r,H_0),$ and hence
the singular spectral shift function $\xis(\lambda; H_r,H_0)$ too, can be considered as an explicitly defined function of the coupling
constant $r.$ In \cite{Az3v5} it was shown that these functions are locally analytic (even locally constant in case of $\xis$) functions of $r$
and that the discontinuities (integer jumps) of these functions can occur only at resonance points $r_0 \in R(\lambda; \set{H_s},V)$ of the path $\set{H_r}.$

Since the spectral shift function $\xi(\lambda; H_r,H_0)$ is point-wise additive (see \cite[Theorem 9.6.5]{Az3v5}),
the latter fact naturally suggests that the size of the jump of $\xi$ and $\xis$
as functions of $r$ for a fixed~$\lambda$ at a resonance point of the path $\set{H_r}$
should depend only on the
resonant at~$\lambda$ operator $H_{r_0}$ and the \emph{direction} of the perturbation
operator~$V.$ It turns out that, indeed, to the triple consisting of a point $\lambda \in \Lambda(H_0,F),$ of an operator $H_{r_0}$ resonant at~$\lambda$ and the direction $V$
one can assign a naturally defined integer number, which is called here \emph{resonance index} of the triple. The main result of this paper
can be formulated as follows.

{\bf Theorem.} {\it The jump of the singular spectral shift function (as a function of the coupling constant) at
a resonance point $r_0$ is equal to the resonance index.}

The resonance index of the triple $(\lambda; H_{r_0},V)$ can defined as follows. Consider the meromorphic operator-valued function
$$
  f_z(s) = (1 + sT_{z}(H_0)J)^{-1}.
$$
For $z = \lambda+i0,$ at a resonance point $s=r_0$ this function has a pole. At the same time, for $z  = \lambda+iy$ with $y>0$ this function cannot have real poles,
implying that the pole $r_0$ of the function $f_{\lambda+i0}(s)$ shifts away from the real axis and possibly splits into some $N_+$ poles in the upper half-plane
and some $N_-$ poles in the lower half-plane. It is shown that the number $N_+-N_-$ does not depend on the choice of the regular point $H_0$ on the line $\set{H_r=H_0+rV}.$
The difference $N_+-N_-$ is the resonance index of the triple $(\lambda; H_{r_0},V).$

It is also shown that the resonance index is equal to the residue (up to an absolute constant) of the pole $s = r_0$ of the meromorphic function
$$
  \mbC \ni s \mapsto \Tr(T_{\lambda+i0}(H_s)J).
$$
In the last section of this paper a non-trivial example of singular spectral shift function is given.

\section{Resonance index}
Given a framed Hilbert space $(\hilb,F),$ we denote by $\clA(F)$ the real Banach space of operators of the form
$V = F^*JF$ with bounded self-adjoint $J.$ The norm of $V = F^*JF$ is by definition $\norm{J}.$
Given a self-adjoint operator $H_0$ on $\hilb,$
we denote by $\clA$ the real affine space of self-adjoint operators
$$
  \clA = H_0 + \clA(F).
$$
The affine space $\clA$ has a topology induced by the topology of $\clA(F).$
We say that a real number~$\lambda$ is \emph{essentially $\clA$-regular}
(or just essentially regular, if there is no danger of confusion), if $\lambda \in \Lambda(H,F)$
(see (\ref{F: def of Lambda}) for definition of $\Lambda(H_,F)$) for some $H \in \clA.$
The set of essentially $\clA$-regular points has full Lebesgue measure.
We denote this set by $\Lambda(\clA,F).$
We say that an operator $H$ from $\clA$ is \emph{resonant at} an essentially regular point~$\lambda,$
if $\lambda \notin \Lambda(H,F);$ otherwise we say that $H$ is \emph{regular at}~$\lambda.$ For a fixed $\lambda,$ the set of resonant at~$\lambda$ operators is a closed meager set
(see \cite[Theorem 4.2.5]{Az3v5}). We denote the set of regular at~$\lambda$ operators by $\Gamma(\lambda) = \Gamma(\lambda; \clA,F).$
Thus, $\Gamma(\lambda)$ is a massive open set.
The set of resonant at~$\lambda$ operators will be denoted by $R(\lambda) = R(\lambda; \clA,F).$
For an analytic curve $\gamma$ in $\clA$ there are two possible options: either $\gamma$ is a subset of the resonance set $R(\lambda; \clA,F),$
or the intersection of $\gamma$ with the resonance set is a discrete set (see \cite[Theorem 4.2.5]{Az3v5}).
In the former case we say that the curve $\gamma$ is \emph{resonant at}~$\lambda,$
in the latter case we say that the curve $\gamma$ is \emph{regular at}~$\lambda.$ 
Similarly one can define resonant at~$\lambda$ and regular at~$\lambda$ real-analytic $n$-dimensional manifolds $\subset \clA.$

If an operator $H \in \clA$ is resonant at $\lambda \in \Lambda(\clA,F),$ then an operator $V \in \clA(F)$ is called a \emph{regularizing operator},
if $H+V$ is regular at~$\lambda.$ If $V \in \clA(F)$ is such that $rV$ is regularizing for some $r \in \mbR$ (that is, if $H+rV$ is regular at~$\lambda$),
then $V$ will be called a \emph{regularizing direction}.

In order to clarify these definitions 
let us consider the case of a finite dimensional Hilbert space $\hilb.$
In this case every number~$\lambda$ is essentially regular. 
An operator $H$ is resonant at~$\lambda$
if and only if~$\lambda$ is an eigenvalue of $H.$
If, for example, $V f = 0,$ where $Hf=\lambda f$ and $f\neq 0,$ then $V$ is not a regularizing
direction. In finite-dimensional case the frame $F$ becomes irrelevant.

Let~$\lambda$ be a fixed essentially regular point and let $\gamma = \set{H_r} \subset H_0 + \clA(F)$
be an analytic path regular at~$\lambda$ (here we are a bit sloppy in distinguishing a path and its image).

In \cite{Az3v5} it is proved (Theorem 9.7.6 and Corollary 9.7.7)
that if the functions
$$
  r \mapsto \xi(\lambda; H_r,H_0) \ \ \text{and} \ \  r \mapsto \xis_\gamma(\lambda; H_r,H_0)
$$ are not analytic at some point $r_0,$ then (1) the point $r_0$ is necessarily a resonance point of the path $\gamma,$
(2) at $r_0$ the functions $\xi(\lambda; H_r,H_0)$ and $\xis(\lambda; H_r,H_0)$ have left and right limits and (3) the jump
of these functions is an integer number.
Unlike $\xi$ and $\xis,$ the function $r \mapsto \xia(\lambda; H_r,H_0)$ is analytic \cite[Proposition 8.2.3]{Az3v5}.
The functions $\xi(\lambda; H_r, H_0)$ and $\xia_\gamma(\lambda; H_r, H_0)$ of the coupling constant $r$ are path additive, and so is $\xis_\gamma(\lambda; H_r, H_0).$
This suggests that the size of a jump at a resonance point $r_0$ should depend only on the resonant at~$\lambda$ operator $H_{r_0}$
and the \emph{direction} of~$V.$ It turns out that this is indeed the case.

Namely, given an essentially regular point $\lambda,$ a resonant at~$\lambda$ operator $H$ and a regularizing direction $V,$
one can define an integer number, which I call \emph{resonance index}, which depends only on $\lambda, H$ and $V$ and definition of which follows. It
turns out that the jump of the singular spectral shift function is equal to the resonance index. This is the main result of this paper.

\subsection{Definition of resonance index}
Let~$\lambda$ be an essentially regular point and let $V \in \clA(F).$ Let $H_0 \in \clA$ be an operator regular at $\lambda,$
that is, $H_0 \in \Gamma(\lambda; \clA,F).$
Let $V = F^*JF.$
Consider the operator
\begin{equation} \label{F: def of fz(s)}
  f_z(s) := f_z(s; H_0, V) := (1 + s T_z(H_0)J)^{-1}
\end{equation}
as a function of $s,$ where $T_z(H_0) = F R_z(H_0) F^*.$ By analytic Fredholm
alternative (see e.g. \cite[Theorem 1.8.2]{Ya}), this function is meromorphic.
When $z$ is non-real, by \cite[Lemma 4.1.4]{Az3v5} the function $f_z(s)$
cannot have real poles. On the other hand, for $z = \lambda \pm i0$ the function $f_z(s)$ may have real poles, which are resonance points
of the path $\mbR \ni s \mapsto H_0+sV.$

Note that the second resolvent identity implies that
\begin{equation} \label{F: def of fz(s) (2)}
  f_z(s; H_0, V) = 1 - s T_z(H_s)J
\end{equation}

\begin{picture}(400,120)
\put(15,100){{\small Poles of $f_{\lambda+iy}(s)$ \ (in $s$-plane):}}

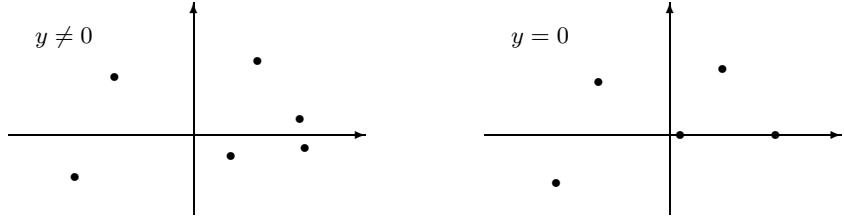
\put(45,24){\circle*{3}}
\put(130,46){\circle*{3}}
\put(132,35){\circle*{3}}
\put(114,68){\circle*{3}}
\put(60,62){\circle*{3}}
\put(104,32){\circle*{3}}

\put(20,40){\vector(1,0){135}}
\put(90,10){\vector(0,1){80}}
\put(30,75){{\small $y \neq 0$}}

\put(227,22){\circle*{3}}
\put(310,40){\circle*{3}}
\put(290,65){\circle*{3}}
\put(243,60){\circle*{3}}
\put(274,40){\circle*{3}}

\put(200,40){\vector(1,0){135}}
\put(270,10){\vector(0,1){80}}
\put(210,75){{\small $y = 0$}}
\end{picture}

Poles of the function $f_z(s)$ under fixed $z$ form a discrete subset of $\mbC.$ Note that
\begin{equation} \label{F: a pole iff an eigenvalue}
  s_0 \text{ is a pole of } f_z(s) \ \Leftrightarrow \ -1/s_0 \text{ is an eigenvalue of } T_0(z)J.
\end{equation}
We say that the pole $s_0$ has multiplicity $N,$ if the corresponding eigenvalue $-1/s_0$ of $T_0(z)J$ has
multiplicity~$N.$ If $\Im z \neq 0,$ then the operator $T_0(z)J$ is trace class, and so in this case the multiplicity of a pole $s_0$ of $f_z(s)$ is equal
to the multiplicity of $-1/s_0$ as a zero of the entire function (called perturbation determinant)
$$
  \mbC \ni s \mapsto \det(1+sT_0(z)J).
$$
When $z$ changes slightly outside of $\mbR,$ the operator $T_0(z)J$ changes slightly in the uniform norm; as a consequence, eigenvalues $-1/s_0$ change slightly
(see e.g. \cite{Kato}).
By (\ref{F: a pole iff an eigenvalue}), this means that poles of $f_z(s)$ change slightly, when $z$ changes slightly.
If a pole $s_0$ has multiplicity $>1,$ then, as $z$ changes slightly, the pole may split into several poles. We say that those poles belong to the same group.
As it follows from stability properties of isolated eigenvalues (see e.g. \cite{Kato}),
the total multiplicity of poles in a group is preserved.

Let $s = r_0$ be a real pole of $f_{\lambda+i0}(s),$ and let $N$ be the multiplicity of this pole.
When $\lambda+i0$ is changed to $\lambda+iy,$
by \cite[Lemma 5.4]{Az3v5}
the pole $r_0$ gets off the real axis and it may split to a group of poles. If $N$ is the multiplicity of the real pole $r_0$,
then in the upper half-plane there appear in total $N_+$ poles (counting multiplicities) in the group
of $r_0$ and in the lower half-plane there appear $N_-$ poles, where
$N_+$ and $N_-$ are some two integers such that
$0 \leq N_+, N_- \leq N$ and $N = N_+ + N_-.$

We shall call poles of a group from $\mbC_+$ (respectively, $\mbC_-$) \emph{up-poles} (respectively, \emph{down-poles}).

\begin{rems} \rm Since spectral measures of operators $T_z(H_0)J$ and $JT_z(H_0)$ coincide,
in the definition of resonance index which follows one can
use the function $$(1 + s J T_z(H_0))^{-1}$$ instead of (\ref{F: def of fz(s)}).
On the other hand, by \cite[(12)]{Az3v5} and \cite[(13)]{Az3v5},
usage of the functions $$(1 + s J T_{\bar z}(H_0))^{-1} \ \text{and} \ (1 + s T_{\bar z}(H_0)J)^{-1}$$
would obviously lead to a change of sign in the definition of the resonance index.
\end{rems}

Recall that property of an operator $H\in \clA$ to be resonant at some point~$\lambda$ does not depend on a path $\set{H_r}$ containing $H.$
\begin{defn} \label{D: res index}
Let~$\lambda$ be an essentially regular point. Let $H_0 \in \clA,$ let $V \in \clA(F)$ and let $H_r = H_0 + rV,$
$r \in \mbR.$ Assume that the line $\set{H_r, r \in \mbR}$ is regular at~$\lambda$ and let $H_{r_0}$
be resonant at $\lambda.$
\emph{Resonance index} of a resonant at~$\lambda$ operator $H_{r_0}$ in regularizing direction $V$
is the number
$$
  \mathrm{ind}_{res} (\lambda; H_{r_0},V) = N_+ - N_-.
$$
That is, the resonance index is equal to the difference of the numbers of up-poles and down-poles in the group of $r_0,$
where poles are counted according to their multiplicity.
\end{defn}
To the best knowledge of the author, this notion is new.

One should justify this definition, since the definition of numbers $N_\pm$ involves operator~$H_0,$
which can be chosen arbitrarily on the line $\set{H_0+rV, r \in \mbR}.$
The next proposition does this by showing that the difference $N_+-N_-$ does not depend on the choice of an operator $H_0$ on the line $\set{H_0 + rV, r \in \mbR}.$
For this, we use the second expression for $f_z(s)$ in (\ref{F: def of fz(s)}) which does not depend on particular choice of~$H_0.$

In fact, one can say more: it turns out that the kernel and the root space corresponding to zero eigenvalue of $1 + r_0 T_{\lambda+i0}(H_0)J$
are invariants of the resonant operator $H_{r_0}$ and a regularizing \emph{direction} $V.$
This will be shown in Corollaries \ref{C: A} and \ref{C: B}.

Before proceeding to the following proposition, we make one remark about indexing of resonant operators.
Usually we denote by $H_0$ an operator which is regular at $\lambda,$ and by $H_{r_0},$ where $r_0\neq 0,$
we denote an operator which is resonant at $\lambda.$ However, since from now on we shall focus on resonant operators,
we shall assume that $H_0$ is resonant at $\lambda.$ For this reason, we use definition (\ref{F: def of fz(s) (2)}) of $f_z(s).$

\begin{prop} \label{P: A}
Let $H_0\in \clA,$ let $V=F^*JF \in \clA(F)$ and let $H_{r} = H_0+rV,$ where $r \in \mbR.$ Let $z \in \mbC \setminus \mbR.$
If a non-zero vector $\psi_{z} \in \clK$ and a number $\alpha_z \in \mbC_+$ are such that for some $r \in \mbR$
$$
  (1-(r+\alpha_z)T_{z}(H_{r})J)\psi_{z} = 0,
$$
then for any other real number $s$ the equality
$$
  (1-(s+\alpha_z)T_{z}(H_{s})J)\psi_{z} = 0
$$
holds. In particular,
$$
  (1-\alpha_z T_{z}(H_{0})J)\psi_{z} = 0.
$$
\end{prop}
\begin{proof} Let
\begin{equation} \label{F: A(-s)z}
  A^{(s)}_{z} = T_{z}(H_{s})J.
\end{equation}
The second resolvent identity implies (see also \cite[(4.7)]{Az3v5}) that for any two real numbers $s,r$
\begin{equation} \label{F: A(-s)z=...}
  A^{(s)}_z  = (1+(s-r)A^{(r)}_z)^{-1}A^{(r)}_z.
\end{equation}
We have
$$
  A^{(r)}_z \psi_z = \frac 1{r + \alpha_z} \psi_z.
$$
It follows that
\begin{equation*} 
  \begin{split}
    A^{(s)}_z \psi_z & = (1+(s-r)A^{(r)}_z)^{-1} A^{(r)}_z \psi_z
    \\ & = \SqBrs{1+(s-r) \cdot \frac 1{r + \alpha_z}} ^{-1} \frac 1{r + \alpha_z} \psi_z = \frac 1{s + \alpha_z}\psi_z.
  \end{split}
\end{equation*}
This completes the proof.
\end{proof}
Main consequence of this proposition is that (a) the eigenvector $\psi_z$ depends only on $H_0$
(which can so far be assumed to be an arbitrary operator from $\clA,$ since $z \notin \mbR$), on the number $z$
and on the \emph{direction} of the operator $V;$
the vector $\psi_z$ does not depend on the choice of an operator $H_s$ on the affine line $\set{H_r, r \in \mbR}$
and (b) the eigenvalue $\frac 1{s+\alpha_z}$ of $H_s$ depends on the choice of $H_s$ but in a simple manner;
more importantly, the half-plane $\mbC_\pm,$ to which eigenvalues of
$T_z(H_r)J$ belong, do not depend on $r.$


The situation becomes more interesting when $z = \lambda\pm i0$ and the operator $H_0$ is not regular at~$\lambda.$
In this case operators $T_{\lambda\pm i0}(H_0)$ do not exist, so that we cannot define $\psi_{\lambda+i0}$ directly.
But we can regularize the operator $H_0$ by adding to it the perturbation $sV$ and then define the eigenvector $\psi_{\lambda+i0}$
using a shifted operator $H_0+sV.$ The eigenvector $\psi_{\lambda+i0}$ does not depend on the choice of $s,$ as the following corollary shows.

\begin{cor} \label{C: A} Let~$\lambda$ be an essentially regular point.
Let $H_0\in \clA$ be resonant at~$\lambda$ operator and let $rV=rF^*JF$ be a regularizing operator,
so that the operator $H_{r} = H_0+rV$ is regular at~$\lambda.$
If a non-zero vector $\psi_{\lambda+i0} \in \clK$ is such that
\begin{equation} \label{F: resonance equation}
  (1-rT_{\lambda+i0}(H_{r})J)\psi_{\lambda+i0} = 0,
\end{equation}
then for any other regularizing operator $sV$ the equality
$$
  (1-sT_{\lambda+i0}(H_{s})J)\psi_{\lambda+i0} = 0
$$
holds.
\end{cor}
\begin{proof} Firstly, we note that since the operator $H_0 = H_{r}-rV$ is resonant at $\lambda,$
such a non-zero vector $\psi_{\lambda+i0}$ necessarily exists (see \cite[Theorem 4.1.11]{Az3v5}).

Since operators $rV$ and $sV$ are regularizing, (using notation (\ref{F: A(-s)z})) operators $A^{(r)}_{\lambda+i0}$ and $A^{(s)}_{\lambda+i0}$ make sense.
By \cite[Theorem 4.1.11]{Az3v5}, the (bounded) operator $(1+(s-r)A^{(r)}_{\lambda+i0})^{-1}$ also exists.

So, the proof follows verbatim the proof of Proposition \ref{P: A} with $z = \lambda+i0$ and $\alpha_z =~0.$
\end{proof}

In conditions of this corollary, the number $0$ is a resonance point of the path $\set{H_r}.$

Since the compact operator $T_{\lambda+iy}(H_s)J$ is not self-adjoint, in general
the root space corresponding to the eigenvalue $s^{-1}$ of $T_{\lambda+i0}(H_s)J$
may have root vectors which are not eigenvectors. The following proposition shows
that the root space also does not depend on $s.$

\begin{prop} \label{P: B}
Let $H_0\in \clA,$ let $V=F^*JF$ and let $H_{r} = H_0+rV,$ where $r \in \mbR.$ Let $z \in \mbC \setminus \mbR$
and let $n \in \set{1,2,\ldots}.$
If a non-zero vector $\psi_{z} \in \clK$ and a number $\alpha_z \in \mbC_+$ are such that for some $r \in \mbR$
$$
  (1-(r+\alpha_z)T_{z}(H_{r})J)^n\psi_{z} = 0,
$$
then for any other real number $s$ the equality
$$
  (1-(s+\alpha_z)T_{z}(H_{s})J)^n\psi_{z} = 0
$$
holds. In particular,
$$
  (1-\alpha_z T_{z}(H_{0})J)^n\psi_{z} = 0.
$$
\end{prop}
\begin{proof} For $n=1$ this is proved in Proposition \ref{P: A}. Assume that the assertion is true for $n=k.$

For $n=k+1,$ we have
$$
  (1-(r+\alpha_z)T_{z}(H_{r})J)(1-(r+\alpha_z)T_{z}(H_{r})J)^k\psi_{z} = 0.
$$
It follows from this and Proposition \ref{P: A} (applied to the vector $(\ldots)^k \psi_z$) that
$$
  (1-(s+\alpha_z)T_{z}(H_{s})J)(1-(r+\alpha_z)T_{z}(H_{r})J)^k\psi_{z} = 0.
$$
Since the operators $T_{z}(H_{s})J$ and $T_{z}(H_{r})J$ commute, it follows that
$$
  (1-(r+\alpha_z)T_{z}(H_{r})J)^k (1-(s+\alpha_z)T_{z}(H_{s})J) \psi_{z} = 0.
$$
By the induction assumption, this implies that
$$
  (1-(s+\alpha_z)T_{z}(H_{s})J)^k (1-(s+\alpha_z)T_{z}(H_{s})J) \psi_{z} = 0.
$$
Proof is complete.
\end{proof}

\begin{cor} \label{C: B} Let~$\lambda$ be an essentially regular point.
Let $H_0\in \clA$ be resonant at~$\lambda$ operator and let $rV=rF^*JF$ be a regularizing operator,
so that the operator $H_{r} = H_0+rV$ is regular at~$\lambda.$
Let $n \in \set{1,2,\ldots}.$
If a non-zero vector $\psi_{\lambda+i0} \in \clK$ is such that
\begin{equation} \label{F: resonance equation of order n}
  (1-rT_{\lambda+i0}(H_{r})J)^n\psi_{\lambda+i0} = 0,
\end{equation}
then for any other regularizing operator $sV$ the equality
$$
  (1-sT_{\lambda+i0}(H_{s})J)^n\psi_{\lambda+i0} = 0
$$
holds.
\end{cor}
Proof of this corollary is similar to the proof of Corollary \ref{C: B}, and therefore it is omitted.

Proposition \ref{P: B} and Corollary \ref{C: B} imply that there is a correctly defined resonance index.

\begin{thm} Definition of the resonance index $\ind_{res}(\lambda; H_0,V)$ does not depend on the choice of a regular operator $H_s$ on the affine line $\set{H_r, r \in \mbR}.$
\end{thm}
\begin{proof} 
Proposition \ref{P: B} implies that 
numbers of eigenvalues of $A_z^{(r)} = T_z(H_{r})J$ (counting multiplicities) in the group of a resonance point $r=0,$
which lie in $\mbC_\pm,$
do not depend on the choice of $r.$
So, the resonance index is well-defined as the difference of the numbers of eigenvalues (counting algebraic multiplicities) in those half-planes.
\end{proof}

The equation (\ref{F: resonance equation of order n}) will be called \emph{resonance equation of order $n$}.
Non-zero solutions
$$
  \psi_{\lambda+i0} = \psi_{\lambda+i0}(n,H_0,V)
$$
of the equation (\ref{F: resonance equation of order n}) will be called \emph{resonance vectors of order $n$} for the triple $(\lambda; H_0,V),$
consisting of an essentially regular point $\lambda,$ an operator $H_0$ resonant at~$\lambda$ and a regularizing direction~$V.$

For a resonant at~$\lambda$ operator $H_0$ with regularizing direction $V$
we denote by
$$
  \Upsilon(\lambda+i0; H_0,V)
$$
the root space of the operator $1-rT_{\lambda+i0}(H_{r})J,$ that is, the finite-dimensional vector space spanned by resonance vectors of all orders.
As shown above, this root space does not depend on $r.$
When $\lambda+i0$ is changed to $\lambda+iy$ with small positive $y,$ the root space changes continuously,
and the corresponding pole splits (more exactly, may split) into a group of non-real poles, but the total multiplicity remains
unchanged. We denote by
$$
  \Upsilon^\uparrow(\lambda+iy; H,V) \ (\text{respectively}, \ \Upsilon^\downarrow(\lambda+iy; H,V))
$$
the direct sum of root spaces corresponding to poles of the group in $\mbC_+$ (respectively, $\mbC_-$).

By definition of resonance index, the equality
\begin{equation} \label{F: ind-res(z)=dim up-dim down}
  \ind_{res}(\lambda; H,V) = \dim \Upsilon^\uparrow(\lambda+iy; H,V) - \dim \Upsilon^\downarrow(\lambda+iy; H,V)
\end{equation}
holds for all small enough $y>0.$

Assume that $H_0$ is resonant at~$\lambda$ and that $V$ is a regularizing direction.
Let us consider the idempotent operator
\begin{equation} \label{F: P(l+i0)=...}
  P(\lambda+i0; H_0,V) = \frac 1{2\pi i} \oint _{C(1/s)} \brs{t - T_{\lambda+i0}(H_s)J}^{-1}\,dt,
\end{equation}
where $C(1/s)$ is a small circle enclosing the eigenvalue $1/s$ of the operator $T_{\lambda+i0}(H_s)J.$
We know, by Corollary \ref{C: B} that the image $\Upsilon(\lambda+i0; H_0,V)$ of this idempotent does not depend on $s.$
Since the compact operator $T_{\lambda+i0}(H_s)J$ is not self-adjoint, the idempotent operator
$P(\lambda+i0; H_0,V)$ in general could depend on $s,$ but it turns out that
\begin{thm} The idempotent operator $P(\lambda+i0; H_0,V),$ defined by formula (\ref{F: P(l+i0)=...}), does not depend on $s.$
\end{thm}
\begin{proof} Let $P_1$ and $P_2$ be two idempotents for two different values $s_1$ and $s_2$ of $s.$
Since, by Corollary \ref{C: B}, these idempotents have the same range $\Upsilon(\lambda+i0; H_0,V),$
we have $P_1P_2 = P_2$ and $P_2P_1 = P_1.$ Since, by (\ref{F: A(-s)z=...}), operators $T_{\lambda+i0}(H_{s_1})J$ and
$T_{\lambda+i0}(H_{s_2})J$ commute, it follows from (\ref{F: P(l+i0)=...}) that $P_1$ and $P_2$ also commute.
It follows that $P_1 = P_2P_1 = P_1P_2 = P_2.$
\end{proof}

%

The following proposition asserts a natural and expected property of the resonance index.

\begin{prop} If~$\lambda$ is an essentially regular point, if $H_0$ is a resonant at~$\lambda$ operator
and if $V$ is a regularizing direction, then
$$
  \ind_{res}(\lambda; H,-V) = - \ind_{res}(\lambda; H,V).
$$
\end{prop}
\begin{proof} A resonant operator $H_0$ can be crossed in two directions along the line $\set{H_r, r \in \mbR}:$ $V$ and $-V.$
In one direction the resonance equation is
($r$ can be complex, but we write $r+\alpha_z$ with real $r$ and complex $\alpha_z$ instead of complex $r$)
\begin{equation} \label{F: up=0}
  (1-(r+\alpha_z)T_{z}(H_{r})J)^n\psi = 0.
\end{equation}
If the resonant operator $H_0$ is now to be crossed in the opposite direction $-V,$ then we have to replace $J$ by $-J$
in the resonance equation to get
$$
  (1+(r+\alpha_z)T_{z}(H_{-r})J)^n\psi = 0.
$$
Clearly, the number $r$ can be chosen so that both $r$ and $-r$ are not resonant.
By Proposition \ref{P: A}, in the last equality we can replace $r$ by $-r$ to get
$$
  (1-(r-\alpha_z)T_{z}(H_{r})J)\psi = 0.
$$
Comparing this with (\ref{F: up=0}) we see that if $r+\alpha_z$ is an up-pole (respectively, down-pole) for the second case, then
$r-\alpha_z$ is a down-pole (respectively, up-pole) for the first case and obviously vice versa.
That is, up-poles (down-poles) in the direction $V$ become down-poles (up-poles) in the direction $-V.$
Clearly, multiplicities of these poles do not change. So, overall the resonance index changes its sign.
This completes the proof.
\end{proof}

\section{Resonance index and singular spectral shift function}
In this section we prove the main result of this paper, which states that
the jump of the singular part $\xis(\lambda; H_r,H_0)$ of the spectral shift function
(considered as a function of the coupling constant $r$ under fixed essentially regular point~$\lambda$)
at a resonance point $H_{r_0}$
of a path $H_r = H_0 + rV$ is equal to the resonance index of the triple $(\lambda, H_{r_0},V).$
\subsection{Resonance index as a residue of a meromorphic function}
For a straight path $H_r = H_0 + rV,$
let us consider the function
$$
  F_z(s) = F_z(s;H_0,V) = \Tr\brs{\frac 1\pi \Im R_z(H_s)V}.
$$
For further use, we note the following equality (see e.g. \cite[(4.8)]{Az3v5})
\begin{equation} \label{F: Fz(s)=...}
  \frac 1\pi \Im R_z(H_s) = \brs{1 + s R_{\bar z}(H_0)V}^{-1} \frac 1\pi \Im R_z(H_0) \brs{1 + sV R_z(H_0)}^{-1}.
\end{equation}

\begin{lemma} \label{L: Im RV=...} The equality
$$
  \frac 1\pi \Im R_z(H_s)V = \frac 1{2\pi i} \cdot \frac 1s \SqBrs{\brs{1 + sR_{\bar z}(H_0)V}^{-1} - \brs{1 + s R_z(H_0)V}^{-1}}.
$$
holds.
\end{lemma}
\begin{proof} By (\ref{F: Fz(s)=...}) we have
$$
  (E) := \frac 1\pi \Im R_z(H_s)V = \brs{1 + s R_{\bar z}(H_0)V}^{-1} \frac 1\pi \Im R_z(H_0) \brs{1 + sV R_z(H_0)}^{-1}V.
$$
It follows that
\begin{equation*}
  \begin{split}
   2\pi i (E) & = \brs{1 + s R_{\bar z}(H_0)V}^{-1} \Brs{R_{z}(H_0) - R_{\bar z}(H_0)} V \brs{1 + s R_z(H_0)V}^{-1}
    \\ & = \brs{1 + s R_{\bar z}(H_0)V}^{-1} R_z(H_0)V \brs{1 + s R_z(H_0)V}^{-1}
    \\ & \qquad - \brs{1 + s R_{\bar z}(H_0)V}^{-1} R_{\bar z}(H_0)V \brs{1 + sR_z(H_0)V}^{-1}
    \\ & = \frac 1s \brs{1 + s R_{\bar z}(H_0)V}^{-1} \SqBrs{1 - \brs{1 + sR_z(H_0)V}^{-1}}
    \\ & \qquad - \frac 1s \SqBrs{1 - \brs{1 + s R_{\bar z}(H_0)V}^{-1}} \brs{1 + sR_z(H_0)V}^{-1}
    \\ & = \frac 1s \SqBrs{\brs{1 + sR_{\bar z}(H_0)V}^{-1} - \brs{1 + s R_z(H_0)V}^{-1}}.
  \end{split}
\end{equation*}
The proof is complete.
\end{proof}

\begin{cor} \label{C: Fz(s)=Tr[(1+RV)(-1)-...]} The equality
$$
  F_z(s) = \frac 1{2\pi i} \cdot \frac 1s \Tr\SqBrs{\brs{1 + sR_{\bar z}(H_0)V}^{-1} - \brs{1 + s R_z(H_0)V}^{-1}}.
$$
holds. In particular, the function $\mbR \ni s \mapsto F_z(s)$ admits a meromorphic continuation to the whole complex plane $\mbC.$
\end{cor}

Note that since $F_z(s)$ takes real values when $s$ is real, the meromorphic continuation of $F_z(s)$ is a symmetric function:
\begin{equation} \label{F: F is symm-c}
  F_z(\bar s) = \overline{F_z(s)}.
\end{equation}

\begin{thm} \label{T: res F(s)=res.ind} Let $z \notin \mbR.$ If $s_0$ is a (necessarily non-real) pole of the meromorphic function
$$
  \mbC \ni s \mapsto F_z(s) = \Tr\brs{\frac 1\pi \Im R_z(H_s)V },
$$
then the equality
$$
  2\pi i \res_{s=s_0} F_z(s) = N_+-N_-
$$
holds, where $N_+$ (respectively, $N_-$) is the multiplicity of $-s_0^{-1}$ as an eigenvalue of $R_{z}(H_0)V$
(respectively, $R_{\bar z}(H_0)V$).
\end{thm}
\begin{proof} Let $C$ be a small circle (oriented anticlockwise) around the pole $s_0,$ such that there are no other poles of $F_z(s)$ on and inside of the circle $C$
except $s_0.$
By Corollary \ref{C: Fz(s)=Tr[(1+RV)(-1)-...]} we have
\begin{equation*}
  \begin{split}
     2 \pi i \res_{s = s_0} F_z(s) & =  \oint_C F_z(s)\,ds
    \\ & = \frac 1{2\pi i} \oint_C \Tr\brs{\frac 1s \SqBrs{\brs{1 + s R_{\bar z}(H_0)V}^{-1} - \brs{1 + s R_z(H_0)V}^{-1}}}\,ds
    \\ & = \frac 1{2\pi i} \Tr \brs{\oint_C \SqBrs{\brs{-s^{-1} - R_{z}(H_0)V}^{-1} - \brs{-s^{-1} - R_{\bar z}(H_0)V}^{-1}}s^{-2}\,ds}.
  \end{split}
\end{equation*}
Making the change of variable $t = -s^{-1}$ we get
\begin{equation*}
  \begin{split}
     2 \pi i \res_{s = s_0} F(s) & = \frac 1{2\pi i} \Tr \brs{\oint_{C_1} \SqBrs{\brs{t - R_{z}(H_0)V}^{-1} - \brs{t - R_{\bar z}(H_0)V}^{-1}}\,dt},
  \end{split}
\end{equation*}
where $C_1$ is a small circle enclosing $t_0 = -s_0^{-1}$ (oriented in anticlockwise direction).
It follows that
\begin{equation} \label{F: res=Tr(P)-Tr(P)}
  \begin{split}
     2 \pi i \res_{s = s_0} F(s) & = \Tr \brs{P(z,t_0) - P(\bar z, t_0)},
  \end{split}
\end{equation}
where $P(z,t_0)$ is the (not necessarily orthogonal) projection onto the root space of the eigenvalue $t_0=-1/s_0$ of $R_z(H_0)V.$
By the Lidskii theorem (see e.g. \cite{GK}), $\Tr(P(z,t_0))$ is equal to the multiplicity of $-s_0^{-1} = t_0$ as an eigenvalue of
$R_z(H_0)V.$ It follows from this and (\ref{F: res=Tr(P)-Tr(P)}) that $2\pi i \res_{s = s_0} F(s) = N_+-N_-.$
\end{proof}

\begin{rems} \rm The resonance index is defined as the difference between the number of up-poles
and down-poles of the function $(1+sT_z(H_0)J)^{-1}.$ Up-poles (respectively, down-poles) of this function are conjugates of down-poles
(respectively, up-poles) of the function $(1+sT_{\bar z}(H_0)J)^{-1}.$ So, the set of poles of the function $F_z(s)$
is the union of poles of $(1+sT_z(H_0)J)^{-1}$ and $(1+sT_{\bar z}(H_0)J)^{-1},$ as it can be seen from (\ref{F: Fz(s)=...}).

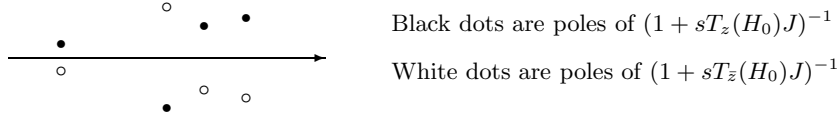
\begin{picture}(400,60)
\put(30,35){\circle*{3}}
\put(30,25){\circle{3}}
\put(100,45){\circle*{3}}
\put(100,15){\circle{3}}
\put(84,42){\circle*{3}}
\put(84,18){\circle{3}}
\put(70,49){\circle{3}}
\put(70,11){\circle*{3}}
\put(10,30){\vector(1,0){120}}
\put(155,40){{\small Black dots are poles of $(1+sT_z(H_0)J)^{-1}$}}
\put(155,22){{\small White dots are poles of $(1+sT_{\bar z}(H_0)J)^{-1}$}}
\end{picture}
\\ It follows that the resonance index can also be defined as the number of black dots in a group in $\mbC_+$ minus the number of white dots in the group in $\mbC_+.$
\end{rems}

\begin{cor} \label{C: residue of real s = 0} Let $\lambda \in \mbR$ be an essentially regular point.
Let the straight line $\set{H_r = H_0+rV, r \in \mbR}$ be regular at~$\lambda.$
The residue of any real pole $s_0$ of the meromorphic function
$$
  \mbC \ni s \mapsto F_{\lambda+i0}(s) = \frac 1\pi \Tr\brs{\Im R_{\lambda+i0}(H_s)V}
$$
is equal to $0.$
\end{cor}
\begin{proof} Clearly,
$$
  \res_{s=s_0} F_{\lambda+i0}(s) = \lim_{y \to 0^+} \res_{s=s_0} F_{\lambda+iy}(s).
$$
If $C$ is a small circle anticlockwise oriented and enclosing $s_0$ such that $s_0$ is the only pole of $F_{\lambda+i0}(s)$ inside and on the circle,
then for small enough $y>0$ the circle $C$ will contain only poles of $F_{\lambda+iy}(s)$ which belong to the group of $s_0.$

So, it is enough to show that for all small enough $y>0$
$$
  \oint_C F_{\lambda+iy}(s)\,ds = 0.
$$
The integral in the left hand side is equal to the sum of residues (times $2\pi i$) of all poles in the group of $s_0.$
Since 
an eigenvalue $-s_0(y)^{-1}$ of $R_{\lambda+iy}(H_0)V$ corresponds to the eigenvalue $-\bar s_0(y)^{-1}$
of $R_{\lambda-iy}(H_0)V$ of the same multiplicity, Theorem \ref{T: res F(s)=res.ind} and Definition \ref{D: res index} complete the proof.
\end{proof}

\subsection{Resonance index and singular spectral shift function}
In the proof of the following theorem we shall use the following contour of integration:
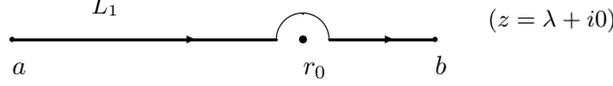
\begin{picture}(400,60)

\put(100,40){{\small $L_1$}}
\put(250,35){{\small $(z=\lambda+i0)$}}

\put(70,30){\circle*{2}}
\put(70,17){{$a$}}

\put(70,30){\vector(1,0){70}}
\put(70,30){\line(1,0){100}}
\put(190,30){\vector(1,0){25}}
\put(190,30){\line(1,0){40}}

\put(230,30){\circle*{2}}
\put(230,17){{$b$}}

\put(180,30){\circle*{3}}
\put(180,30){\oval(20,20)[t]}
\put(180,17){{$r_0$}}
\end{picture}

\begin{thm} \label{T: G(l+i0)=xia} Let~$\lambda$ be an essentially regular point.
Let $\gamma = \set{H_r=H_0+rV, r \in [a,b]}$ be a straight path with only one resonance point $r_0 \in (a,b)$ at~$\lambda.$
For $r \in [a,b] \setminus \set{r_0}$ let
$$
  G_{\lambda+i0}(r) = \int_{L_1} F_{\lambda+i0}(s) \,ds = \frac 1\pi \int_{L_1} \Tr\brs{\Im R_{\lambda+i0}(H_s)V}\,ds,
$$
where the contour $L_1$ of integration goes straight from $a$ to $r,$ but,
in case $r > r_0,$ circumvents the resonance point $r_0$ in the upper half-plane, as shown in the picture above,
in such a way that there are no other poles between the half-circle and the real axis except~$r_0.$
Then the function $G_{\lambda+i0}(s)$ admits a single-valued analytic continuation to a neighbourhood of the interval $[a,b].$
Restriction of this analytic continuation to the interval $[a,b]$ coincides with the value of the absolutely continuous part of the spectral shift function at $\lambda:$
$$
  G_{\lambda+i0}(r) = \xia_\gamma(\lambda; H_r,H_0).
$$
\end{thm}
\begin{proof} 
By Corollary \ref{C: residue of real s = 0}, the integral of the function $F_{\lambda+i0}(s)$
over a small circle around $r_0$ is equal to zero.
It follows that $G_{\lambda+i0}(s)$
admits a single-valued analytic continuation to a neighbourhood of $[a,b]$ with only possible singularity at~$r_0.$

Further, for real $r < r_0,$ the function $G_{\lambda+i0}(r)$ coincides with $\xia(\lambda; H_r,H_0),$ since for all such $r$
\begin{equation*}
  \begin{split}
    \Phia_{H_r}(V)(\lambda) & = \Tr(\euE_\lambda(H_r) V \euE_\lambda^\diamondsuit(H_r))
    \\ & = \Tr(\euE_\lambda^\diamondsuit(H_r) \euE_\lambda(H_r)V)
    \\ & = \frac 1\pi \Tr (\Im R_{\lambda+i0}(H_r)V),
  \end{split}
\end{equation*}
where definition \cite[(8.6)]{Az3v5} of $\Phia_{H_r}(V)(\lambda)$ and \cite[(5.5)]{Az3v5} were used.
Since by \cite[Proposition 9.10]{Az3v5}
the function $r \mapsto \xia(\lambda; H_r,H_0)$
admits analytic continuation to a neighbourhood of $[a,b],$ it follows that so does $G_{\lambda+i0}(s)$ and that they coincide.
\end{proof}

Note that since $\xia$ and $\xi$ are pointwise path-additive, the assumption
that there is only one resonant point is not essential at all.

In the proof of the following theorem (which is the main result of this paper)
we shall need the following two contours of integration.

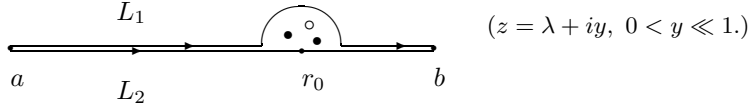
\begin{picture}(400,60)

\put(200,35){{\small $(z=\lambda+iy, \ 0<y\ll 1.)$}}

\put(60,40){$L_1$}

\put(60,10){$L_2$}

\put(20,28){\line(1,0){160}}
\put(20,28){\vector(1,0){50}}

\put(20,29){\circle*{2}}
\put(20,15){{$a$}}

\put(20,30){\vector(1,0){70}}
\put(20,30){\line(1,0){95}}
\put(145,30){\vector(1,0){25}}
\put(145,30){\line(1,0){35}}

\put(180,29){\circle*{2}}
\put(180,15){{$b$}}

\put(125,34){\circle*{3}}
\put(133,38){\circle{3}}
\put(136,32){\circle*{3}}
\put(130,28){\circle*{2}}
\put(130,30){\oval(30,30)[t]}
\put(130,15){{$r_0$}}

\end{picture}

\begin{thm} \label{T: Phis = res-index} Let~$\lambda$ be an essentially regular point. Let $H_r=H_0+rV,$ $a \leq r \leq b,$
and let $r_0$ be a resonance point of this path at~$\lambda.$
The jump of the singular part of the spectral shift function at a resonance point $r_0$
is equal to the resonance index of $H_{r_0}.$ In other words,
$$
  \Phis_{H_{r_0}}(V)(\lambda) = \ind_{res}(\lambda; H_{r_0}, V) \cdot \delta.
$$
\end{thm}
\begin{proof} Since $\xi$ and $\xia$ are pointwise path additive, we can assume that $r_0$ is the only resonance point of the path $\set{H_r}.$
For the same reason, the number $a$ can be replaced by any number between $a$ and $r_0,$ and the number $b$ can be
replaced by any number between $r_0$ and $b.$

When $\lambda +i0$ is perturbed to $\lambda+iy$ with small $y>0,$ the pole $s=r_0$ of the function $F_{\lambda+i0}(s)$
get off the real axis and, in general, splits into a group of poles in $\mbC_\pm$ near~$r_0.$
(Since the function $F_z(s)$ is symmetric for all $z,$ the set of poles of this function is symmetric with respect to the real axis, but we need only those
which lie in $\mbC_+$ because of the choice of the path of integration which circumvents the poles from above).
Let $L_2$ be a straight line which connects $a$ and $b$ and let $L_1$ be a line which also connects $a$ and $b$
but circumvents the poles of the $r_0$-group from above, as shown in the picture above. By \cite[Lemma 4.1.4]{Az3v5},
the contours $L_1$ and $L_2$ do not pass through the poles of the function $F_{\lambda+iy}(\cdot).$

We have the equality
$$
  \int_{L_2} F_{\lambda+iy}(s) \,ds = \int_{L_1} F_{\lambda+iy}(s) \,ds + I(\lambda+iy),
$$
where $I(\lambda+iy)$ is the integral over the half-circle which encloses all the poles of the group in $\mbC_+.$
By Theorem \ref{T: res F(s)=res.ind}, this integral is independent of small enough $y>0,$
and it is equal to the resonance index of $H_{r_0}.$ Clearly,
$$
  \lim_{y \to 0^+} \int_{L_1} F_{\lambda+iy}(s) \,ds = \int_{L_1} F_{\lambda+i0}(s) \,ds,
$$
which, by Theorem \ref{T: G(l+i0)=xia}, is equal to $\xia(\lambda; H_b,H_a).$
On the other hand, by \cite[Proposition 9.5.4]{Az3v5}, the integral over $L_2$ is equal to the smoothed spectral shift function
$\xi(\lambda+iy; H_b,H_a).$ So, by  \cite[Lemma 9.5.2]{Az3v5} and \cite[Theorem 9.6.1/Definition 9.6.2]{Az3v5},
the integral over $L_2$ converges to $\xi(\lambda+i0; H_b,H_a) = \xi(\lambda; H_b,H_a)$ as $y \to 0.$ Hence, for any $a<r_0$ and any $b>r_0$
$$
  \xi(\lambda; H_b,H_a) = \xia_\gamma(\lambda; H_b,H_a) + \ind_{res}(\lambda; H_{r_0},V).
$$
This completes the proof.
\end{proof}
If there are several resonance points $r_1, r_2, \ldots, r_N$ in the interval $[a,b],$ then
$$
  \xi(\lambda; H_b,H_a) = \xia_\gamma(\lambda; H_b,H_a) + \sum_{j=1}^N\ind_{res}(\lambda; H_{r_j},V).
$$
That is,
$$
  \xis(\lambda; H_b,H_a) = \sum_{j=1}^N\ind_{res}(\lambda; H_{r_j},V).
$$

\begin{rems} \rm It is known that outside of the essential spectrum the spectral flow
(as defined by J.\,Phillips in \cite{Ph96CMB,Ph97FIC}) and the Lifshits-Krein spectral shift functions
coincide, see e.g. \cite{ACS,ACDS}. Theorem \ref{T: Phis = res-index} shows that the resonance index coincides with spectral flow
outside the essential spectrum, but unlike the spectral flow it also makes sense inside the essential spectrum. As such,
the resonance index can be considered as a proper generalization of spectral flow to the essential spectrum.
\end{rems}

\subsection{Large coupling constant limit}
Let~$\lambda$ be an essentially regular point, let $H_0 \in \clA$ and let $V$ be a regularizing direction
of finite rank.
In this subsection we show that the limit
$$
  \lim_{r \to \infty} \xi(\lambda; H_r, H_{-r})
$$
exists and is equal to the signature of the operator~$V.$ Since a.e.~$\lambda$ is regular for $H_0,$
it will follow that the limit
$$
  \lim_{r \to \infty} \xi(\cdot; H_r, H_{-r})
$$
exists a.e. and is constant. This will give a new proof (to the best of my knowledge) of a well-known fact
that the spectral shift function is constant at large coupling constant limit (see \cite{SimTrId2} and references in it).
Also, one has to note that this assertion is proved in the literature for the case
of one-dimensional perturbations (to the best of my knowledge),
and its generalization to the case of finite-rank perturbations seems to be not straightforward.

\begin{thm} \label{T: a beauty theorem} Let~$\lambda$ be an essentially regular point of the pair $(\clA,F).$
Let $H_0$ be a self-adjoint operator and let $V = F^*JF \in \clA(F)$ be a self-adjoint finite-rank operator such that
the line $\set{H_r = H_0+rV}$ does not lie in the resonance set $R(\lambda; \clA,F).$
Then the limit
$$
  \lim_{r \to \infty} \xi(\lambda; H_r, H_{-r})
$$
exists and is equal to the signature of the operator~$V.$
\end{thm}

This theorem implies the following well-known (in case of $\rank(V) = 1$) result (see e.g. \cite{SimTrId2}).
\begin{cor} Let $H_0$ be a self-adjoint operator and let $V$ be a self-adjoint finite-rank operator.
For any finite interval $\Delta \subset \mbR$
the limit
$$
  \lim_{r \to \infty} \int_\Delta \xi(\lambda; H_r, H_{-r})\,d\lambda
$$
exists and is equal to $k\abs{\Delta} ,$ where
$k$ is the signature of $V$ and $\abs{\Delta}$ is the Lebesgue measure of $\Delta.$
\end{cor}
Indeed, for any finite rank operator $V$ there exists a frame $F$ such that $V \in \clA(F),$
and every~$\lambda$ from the set of full Lebesgue measure $\Lambda(H_0,F)$ satisfies conditions of Theorem \ref{T: a beauty theorem}.

We shall need the following well-known theorem.
\begin{thm} \label{T: Krein's thm} Let $H$ be a self-adjoint operator on $\hilb,$ and let $V$ be a finite rank self-adjoint operator on $\hilb.$
If $\Im z > 0,$ then the operator $R_z(H)V$ has $\rank(V_+)$ eigenvalues (counting eigenvalues) in the upper half-plane $\mbC_+,$
it has $\rank(V_-)$ eigenvalues (counting eigenvalues) in the lower half-plane $\mbC_-$ and it has no eigenvalues on the real line except zero.
\end{thm}
A proof of this theorem can easily be derived from assertion ($\alpha$) in subsection 4 of \cite[\S 3]{Kr53MS}.
Concerning a more exact position of eigenvalues in the upper/lower half-planes see also \cite{Wi67JMAA}.

Since the spectral measures of operators $R_z(H)V$ and $T_z(H)J$ coincide, this theorem holds for the operator $T_z(H)J$ as well.

\medskip

{\it Proof of Theorem \ref{T: a beauty theorem}} \
%

(A) Let $L$ be a closed contour of integration in the complex coupling constant plane, 
which consists of two pieces: a piece $L_2,$ which goes straight from $-R$ to $R$ and $L_1$ which returns back
from $R$ to $-R$ via a half-circle of radius $R$ in the upper half-plane,
where $R$ is sufficiently large so that the disk $\set{z \in \mbC \colon \abs{z} < R}$ contains
all eigenvalues of $T_{\lambda+iy}(H_0)J$ for all small enough $y>0.$

Let $F_z(s) = \frac 1\pi \Tr (T_{z}(H_0)J).$ By Theorem \ref{T: res F(s)=res.ind}, for all small enough $y>0$ the integral
$$
  \int_L F_z(s)\,ds
$$
is equal to $N_+-N_-,$ where $N_\pm$ is the total number of eigenvalues of $R_z(H)V$ in $\mbC_\pm.$
Hence, by Theorem \ref{T: Krein's thm} the integral is equal to the signature of the operator~$V.$

(B) It follows from the proof of Theorem \ref{T: Phis = res-index} that as $y \to 0^+$ the integral over $L_1$ goes to $\xi(\lambda; H_{-R},H_R).$
Also, as $y \to 0^+,$ the integral over $L_2$ converges.

(C) Here we show that the integral of $F_{\lambda+i0}(s)$ over the arc $L_2$ goes to zero as $R \to \infty.$
Since, by (A), the integral over $L$ is constant and is equal to the signature of $V,$
this will imply that the integral of $F_{\lambda+i0}(s)$ over $L_1$ converges to the signature of $V$ as $R \to \infty,$
thus completing the proof.

The operator $T_{\lambda+i0}(H_0)J$ has a finite number of non-zero eigenvalues $\lambda_j = -1/s_j, \ j =1, \ldots, N.$
It follows that the operator
$$
  \brs{1 + s T_z(H_0)J}^{-1}-1
$$
has non-zero eigenvalues
${s}/\brs{s_j-s}, \ j =1, \ldots, N.$
So, the Lidskii theorem (see e.g. \cite{GK,SimTrId2}) implies that
$$
  \Tr\brs{\brs{1 + s T_z(H_0)J}^{-1}-1} = \sum_{j=1}^N \frac{s}{s_j-s}.
$$
Similarly,
$$
  \Tr\brs{\brs{1 + s T_{\bar z}(H_0)J}^{-1}-1} = \sum_{j=1}^N \frac{s}{\bar s_j-s}.
$$
Combining the last two equalities we get that
$$
  \Tr\SqBrs{\brs{1 + sT_{\bar z}(H_0)J}^{-1} - \brs{1 + s T_z(H_0)J}^{-1}} = \sum_{j=1}^N \frac{s(s_j-\bar s_j)}{(s-s_j)(s-\bar s_j)}.
$$
This equality implies that the trace in the left hand side converges to zero uniformly with respect to $s \in L_2,$ as $R = \abs{s} \to \infty.$
It now follows from Corollary \ref{C: Fz(s)=Tr[(1+RV)(-1)-...]} that the integral of $F_{\lambda+i0}(s)$ over the arc $L_2$
goes to zero as $R \to \infty.$ Proof is complete.
$\Box$

\begin{cor} Let $H_0,$ $V$ and~$\lambda$ be as in Theorem \ref{T: a beauty theorem}. The limit
$$
  \xia(\lambda; H_{\infty},H_{-\infty}) := \lim_{r \to \infty} \xia(\lambda; H_r,H_{-r})
$$
exists and its value is an integer number. This integer is equal to the number of eigenvalues
of $T_{\lambda+i0}(H_0)J$ in $\mbC_+$ minus the number of eigenvalues of $T_{\lambda+i0}(H_0)J$ in $\mbC_-.$
Also,
$$
  \xis(\lambda; H_{\infty},H_{-\infty}) = \sum_{r \in \mbR} \ind_{res}(\lambda; H_r,V).
$$
\end{cor}

\section{A non-trivial example of singular spectral shift function}
\label{S: non-trivial example}
Let $H$ and $V$ be two self-adjoint operators on a Hilbert space $\hilb.$ By definition, the pair $(H,V)$ is irreducible,
if the only non-zero invariant subspace for both of these operators is $\hilb.$

By a non-trivial example of singular spectral shift function we mean a pair $(H,V)$ of irreducible self-adjoint operators,
such that the restriction of the singular part of the spectral shift function $\xis(\lambda; H_0+V,H_0)$ to the absolutely continuous (or essential) spectrum
of $H$ is non-zero. Note that when $V$ is trace-class, all operators $H_r = H_0+rV$ have the same absolutely continuous (and essential) spectrum.
Reducible pairs $(H,V)$ with non-zero singular spectral shift function on absolutely continuous spectrum are trivial to construct.

Let $K$ be a Cantor subset (see e.g. \cite{Nat}) of $[-1,1]$ symmetric with respect to $0$ and such that the Lebesgue measure of $K$ is equal to $1.$
Let $\chi_U$ be the characteristic function of the open set $U = [-1,1] \setminus K$ and let
$$
  F(x) = \int_0^x \chi_U(t)\,dt, \ \ x \in \mbR.
$$
The function $F(x)$ is continuous on $\mbR$ and
\begin{equation} \label{F: F'(x)=0}
  F'(x) = 0 \ \text{for a.e.} \ x \in K.
\end{equation}

Let $H_0$ be the operator of multiplication by $x$ in $\hilb = L_2([-1,1], dF(x)).$ Let $v = 1$ and let $V = \scal{v}{\cdot}v.$
Obviously, the operator $H_0$ is a continuous self-adjoint operator and its spectrum is $[-1,1].$

\begin{lemma} The pair $H_0$ and $V$ is irreducible.
\end{lemma}
\begin{proof} 
(A) If $\clK \neq \set{0}$ is an invariant subspace for both $H_0$ and $V,$ then $\clK$
is not orthogonal to $v.$

Proof. Let $0 \neq u \in \clK.$ It follows that $H_0^n u \in \clK$ for all $n=0,1,2,\ldots$
If all these vectors are orthogonal to $v,$ then $u$ is orthogonal to all vectors $H_0^n v = x^n.$
Since the set of polynomials is dense in $L_2([-1,1],dF),$
it follows that $u=0.$ This contradiction shows that $\clK$ is not orthogonal to $v.$

(B) If $\clK \neq \set{0}$ is an invariant subspace for both $H_0$ and $V,$ then it follows from (A) that $\clK$ contains $v = 1,$
and consequently, it contains all polynomials. It follows that $\clK = \hilb.$
\end{proof}

\begin{lemma} \label{L: p.v.int is non-zero} The principal value integral
$$
  \text{p.v.} \int_\mbR \frac{dF(x)}{x-\lambda}
$$
exists and is non-zero for a.e. $\lambda \in K.$
\end{lemma}
\begin{proof} This integral exists for a.e. $\lambda \in \mbR$ by \cite[Theorem 1.2.5]{Ya}.
It follows from (\ref{F: F'(x)=0}) that if the integral above is zero on a subset of $K$ of positive Lebesgue measure, then
by the Luzin-Privalov theorem (see e.g. \cite[Theorem 1.2.1]{Ya}) the Cauchy-Stieltjes transform of $F$ must be zero,
and hence so must be the function~$F.$
This contradiction completes the proof.
\end{proof}

\begin{thm} Let $H_0$ and $V$ be as above and let $H_r = H_0+rV.$ For all large enough $r>0$
the singular part of the spectral shift function for the pair $(H_0,H_r)$ is non-zero on
the absolutely continuous spectrum $[-1,1]$ of the operator $H_r$ as an element of $L_1(-1,1).$
\end{thm}
\begin{proof} We shall show that the set of those $\lambda \in K,$ for which resonance index $\ind_{res}(\lambda; H_r,V)$
is equal to $1$ for some large enough $r,$ has positive Lebesgue measure. By Theorem \ref{T: Phis = res-index}, this will complete the proof.

Let us consider the function
$$
  f_z(r;H_0,V) = 1 - rT_z(H_r)J.
$$
Since we are interested in poles of this function, 
we can replace it by
$$
  f_z(r;H_0,V) = 1 - rR_z(H_r)V.
$$
By \cite[(6.7.3)]{Ya} (see also \cite[(8.14)]{Az3v5}), poles of this function are roots of the equation
$$
  1+r\scal{v}{R_z(H_0)v} = 0.
$$
By \cite[Theorem 1.2.5]{Ya} and (\ref{F: F'(x)=0}), for a.e. $\lambda \in K$ this equation for $z = \lambda +i0$ can be rewritten as
$$
  1 + r \int_{-1}^1 \frac{dF(x)}{x-\lambda} = 0,
$$
where the integral is the principal value integral.
So, as $\Im z \to 0$ the pole of the function $f_z(s)$ goes to the real number
$$
  r = - \brs{p.v. \int_{-1}^1 \frac{dF(x)}{x-\lambda}}^{-1}.
$$
By Lemma \ref{L: p.v.int is non-zero}, the right hand side is finite for a.e. $\lambda \in K.$
This means that $\ind_{res}(\lambda,H_r,V) = 1$ for a.e. $\lambda \in K.$
Since obviously the p.v. integral as a function of~$\lambda$ is odd, there are positive poles $r$ for a set of $\lambda \in K$
of Lebesgue measure $\frac {\abs{K}}2 = \frac 12.$ This completes the proof.
\end{proof}
It can be seen from the proof that for a.e. $\lambda \in K$
$$
  \lim_{r \to \infty} \xis(\lambda; H_r,H_0) = \chi_K(\lambda).
$$
On the other hand, for a finite $r$ a Borel support of the function $\xis(\lambda; H_r,H_0)$ has even more complicated nature.

\rndef{\emph}[1]{{\it #1}}

\mathsurround 0pt
\ndef{\AndSoOn}{$\dots$}


\end{document}